\documentclass[10pt,a4paper]{article}
\usepackage{latexsym,amssymb,amsmath,mathrsfs,amsthm}
\usepackage{upgreek}
\usepackage{color}
\usepackage{wasysym}
\usepackage{bm}
\usepackage[T1]{fontenc}
\usepackage[anchorcolor=blue,%
 bookmarks=true,%
 bookmarksnumbered=true,%
]{hyperref}
\usepackage{pst-all}
\usepackage{graphicx}
\newtheorem{thm}{Theorem}[section]
\newtheorem{prop}[thm]{Proposition}
\newtheorem{cor}[thm]{Corollary}
\newtheorem{lem}[thm]{Lemma}

\newtheorem{defn}[thm]{Definition}
\newtheorem{remark}[thm]{Remark}
\newtheorem{example}[thm]{Example}
\newtheorem{assumption}{Assumption}

\makeatletter \@addtoreset{equation}{section} \makeatother

\renewcommand{\P}{\mathbb{P}}
\newcommand{\E}{\mathbb{E}}

\newcommand{\R}{\mathbb{R}}

\newcommand{\N}{\mathbb{N}}

\newcommand{\e}{\mathrm{e}}

\renewcommand{\d}{\mathrm{d}}
\newcommand{\m}{\mathfrak{m} }

\newcommand{\1}{{\bf 1}}
\newcommand{\eps}{{\varepsilon}}

\title{\large\bf {\boldmath$(1,p)$}-Sobolev spaces based on strongly local Dirichlet forms}
\author{
Kazuhiro Kuwae\thanks{Department of Applied Mathematics, Fukuoka University,
Fukuoka 814-0180, Japan ({\sf kuwae@}{\sf fukuoka-u.ac.jp}). Supported in part by JSPS Grant-in-Aid for Scientific Research (S) (No. 22H04942) and fund (No.~215001) from the Central Research Institute of Fukuoka University.}
}
\date{}
\begin{document}
\maketitle

\begin{abstract}
In the framework of quasi-regular strongly local Dirichlet form $(\mathscr{E},D(\mathscr{E}))$ on $L^2(X;\m)$ admitting 
minimal $\mathscr{E}$-dominant measure $\mu$, 
we construct a natural $p$-energy functional $(\mathscr{E}^{\,p},D(\mathscr{E}^{\,p}))$ on $L^p(X;\m)$
and $(1,p)$-Sobolev space 
$(H^{1,p}(X),\|\cdot\|_{H^{1,p}})$ for $p\in]1,+\infty[$. 
In this paper, we establish the Clarkson type inequality for $(H^{1,p}(X),\|\cdot\|_{H^{1,p}})$. As a consequence, $(H^{1,p}(X),\|\cdot\|_{H^{1,p}})$ is a uniformly convex Banach space, hence it is reflexive. 
Based on the reflexivity of $(H^{1,p}(X),\|\cdot\|_{H^{1,p}})$, 
we prove that (generalized) normal contraction operates on $(\mathscr{E}^{\,p},D(\mathscr{E}^{\,p}))$, which has been shown in the case of various concrete settings, but has not been 
proved for such general framework.  Moreover, we prove that 
$(1,p)$-capacity ${\rm Cap}_{1,p}(A)<\infty$ for open set $A$ 
admits an equilibrium potential $e_A\in D(\mathscr{E}^{\,p})$ with $0\leq e_A\leq 1$ $\m$-a.e. and $e_A=1$ $\m$-a.e.~on $A$. 
%Our proof relies on the theory of first order 
%vector space calculus for Dirichlet space, which was initiated by Gigli for RCD-spaces and developed by Braun 
%for quasi-regular strongly local Dirichlet spaces.  
%
%Under the weak Bakry-\'Emery curvature lower bounds condition to $(\mathscr{E},D(\mathscr{E})$, we develop the properties of 
%$(\mathscr{E}^p,D(\mathscr{E}^{\,p}))$ and $(H^{1,p}(X),\|\cdot\|_{H^{1,p}})$. 
\end{abstract}

%%%%%%%%%%%%%%%%%%%%%
%%%%%%%%%%%%%%%%%%%%%
%%%%%%%%%%%%%%%%%%%%%
\section{Statement of Main Results}\label{sec:StatementMain}
%%%%%%%%%%%%%%%%%%%%%
%\subsection{Framework}\label{subsec:Frame}
Let $(X,\tau)$ be a topological Lusin space, i.e., a continuous injective image of a Polish 
space, endowed with a $\sigma$-finite Borel measure $\m$ on $X$ with full topological support. 
Let $(\mathscr{E},D(\mathscr{E}))$ be a quasi-regular 
symmetric strongly local Dirichlet space on $L^2(X;\m)$ 
and $(P_t)_{t\geq0}$ the associated symmetric sub-Markovian strongly continuous semigroup on $L^2(X;\m)$ (see \cite[Chapter IV, Definition~3]{MR} for the quasi-regularity and see \cite[Theorem~5.1(ii)]{Kw:func} for the strong locality). 
Then there exists an $\m$-symmetric special standard process ${\bf X}=(\Omega, X_t, \P_x)$ 
associated with $(\mathscr{E},D(\mathscr{E}))$, i.e. for $f\in L^2(X;\m)\cap \mathscr{B}(X)$, 
$P_tf(x)=\E_x[f(X_t)]$ $\m$-a.e.~$x\in X$ (see \cite[Chapter IV, Section~3]{MR}).

It is known that for $u,v\in D(\mathscr{E})\cap L^{\infty}(X;\m)$ there exists a unique signed finite Borel 
measure $\mu_{\langle u,v\rangle}$ on $X$ such that 
\begin{align*}
2\int_X\tilde{f}\d\mu_{\langle u,v\rangle}=\mathscr{E}(uf,v)+\mathscr{E}(vf,u)-\mathscr{E}(uv,f)\quad \text{ for }\quad u,v\in D(\mathscr{E})\cap L^{\infty}(X;\m).
\end{align*}
Here $\tilde{f}$ denotes the $\mathscr{E}$-quasi-continuous $\m$-version of $f$ (see \cite[Theorem~2.1.3]{FOT}, \cite[Chapter IV, Proposition~3.3(ii)]{MR}). 
We set $\mu_{\langle f\rangle}:=\mu_{\langle f,f\rangle}$ for $f\in D(\mathscr{E})\cap L^{\infty}(X;\m)$. 
Moreover, 
for $f,g\in D(\mathscr{E})$, there exists a signed finite measure $\mu_{\langle f,g\rangle}$ on $X$ such that 
$\mathscr{E}(f,g)=\mu_{\langle f,g\rangle}(X)$, hence $\mathscr{E}(f,f)=\mu_{\langle f\rangle}(X)$. 
A Borel measure $\mu$ on $X$ is 
said to be \emph{$\mathscr{E}$-dominant} if it is $\sigma$-finite and $\mu_{\langle f\rangle}\ll \mu$ for all 
$f\in D(\mathscr{E})$, and it is said to be \emph{minimal $\mathscr{E}$-dominant} if it is $\mathscr{E}$-dominant and $\mu\ll\nu$ for every $\mathscr{E}$-dominant Borel measure $\nu$ on $X$. 
Any two minimal $\mathscr{E}$-dominant measure are mutually equivalent. 
The existence of minimal $\mathscr{E}$-dominant measure $\mu$ is shown by Nakao~\cite[Lemma~2.2]{Na}, which is given by a finite smooth measure 
$\mu:=\sum_{k=1}^{\infty}\frac{\mu_{\langle f_k\rangle}}{2^k(1+\mu_{\langle f_k\rangle}(X))}$ for some countable $\mathscr{E}_1^{1/2}$-dense subcollection 
$\{f_k\}$ of $D(\mathscr{E})$.  
The class of $f\in D(\mathscr{E})$ for which $\mu_{\langle f\rangle}$ is minimal $\mathscr{E}$-dominant is dense in $D(\mathscr{E})$ (see 
\cite[Proposition~2.7]{HinoFractal}), and  
every minimal $\mathscr{E}$-dominant measure does not charge $\mathscr{E}$-exceptional sets (see \cite[Proposition~2.18]{DelloSuzuki:LademacherSobolevToLip}). Given any $\mathscr{E}$-dominant $\mu$, there exists a unique symmetric bilinear form $\Gamma_{\mu}:D(\mathscr{E})_e\times D(\mathscr{E})_e\to L^1(X;\mu)$ such that for every $f,g\in D(\mathscr{E})_e$,
\begin{align*}
\mathscr{E}(f,g)=\int_X\Gamma_{\mu}(f,g)\d\mu
\end{align*}
(see \cite[after Definition~1.13]{Braun:Tamed2021}).
Hereafter, we fix a minimal $\mathscr{E}$-dominant measure $\mu$. 
The calculus rules from \cite[Proposition~1.12]{Braun:Tamed2021} transfer accordingly to $\Gamma_{\mu}$ at the $\mu$-a.e.~level
as in the following: 
Take $u,v\in D(\mathscr{E})$. Then 
\begin{enumerate}
\item[(1)]\label{item:1} {\bf (Parallelogram law):} $\Gamma_{\mu}(u+v)+\Gamma_{\mu}(u-v)=2\Gamma_{\mu}(u)+2\Gamma_{\mu}(v)$ $\mu$-a.e.
\item[(2)]\label{item:2} {\bf (Strong local property):} $\1_G\Gamma_{\mu}(u,v)=0$ $\mu$-a.e. for all open sets $G\subset X$ on which $u$ is constant.
\item[(3)]\label{item:3} {\bf (Cauchy-Schwarz inequality):} We have $|\Gamma_{\mu}(u,v)|\leq\Gamma_{\mu}(u)^{\frac12}\Gamma_{\mu}(v)^{\frac12}$ $\mu$-a.e. In particular, $\Gamma_{\mu}(u+v)^{\frac12}\leq\Gamma_{\mu}(u)^{\frac12}+\Gamma_{\mu}(v)^{\frac12}$ $\mu$-a.e. and $|\Gamma_{\mu}(u)^{\frac12}-\Gamma_{\mu}(v)^{\frac12}|\leq\Gamma_{\mu}(u-v)^{\frac12}$ $\mu$-a.e.
\item[(4)]\label{item:4} {\bf (Truncation Property):} $\Gamma_{\mu}(u\lor v)=\1_{\{\tilde{u}\geq \tilde{v}\}}\Gamma_{\mu}(u)+\1_{\{\tilde{u}<\tilde{v}\}}\Gamma_{\mu}(v)$ $\mu$-a.e. and \\
$\Gamma_{\mu}(u\land v)=\1_{\{\tilde{u}\geq \tilde{v}\}}\Gamma_{\mu}(v)+\1_{\{\tilde{u}< \tilde{v}\}}\Gamma_{\mu}(u)$ $\mu$-a.e. 
\item[(5)]\label{item:5} {\bf (Contraction property):}  $\Gamma_{\mu}(0\lor u\land1)\leq\Gamma_{\mu}(u)$ $\mu$-a.e.
\item[(6)]\label{item:6} {\bf (Leibniz rule):}  $\Gamma_{\mu}(uv,w)=\tilde{u}\Gamma_{\mu}(v,w)+\tilde{v}\Gamma_{\mu}(u,w)$ $\mu$-a.e. for $u,v\in D(\mathscr{E})\cap L^{\infty}(X;\mu)$ and $w\in D(\mathscr{E})$.  
\item[(7)]\label{item:7} {\bf (Chain rule):} Let $\phi\in C^1(\R)$ with $\phi(0)=0$ and bounded derivative $\phi'$. Then $\phi(u)\in D(\mathscr{E})$ and $\Gamma_{\mu}(\phi(u),v)=\phi'(\tilde{u})\Gamma_{\mu}(u,v)$ $\mu$-a.e.
\end{enumerate}
We say that $(X,\mathscr{E},\m)$ or simply $\mathscr{E}$ admits a
 carr\'e-du-champ $\Gamma$ if $\m$ is $\mathscr{E}$-dominant in which case we abbreviate $\Gamma_{\m}$ by $\Gamma$ and term it the \emph{carr\'e du champ}, then $\m$ is already \emph{minimal} $\mathscr{E}$-dominant according to the remark after \cite[Definition~3.4]{DelloSuzuki:LademacherSobolevToLip}. 

Fix $p\in [1,+\infty[$. 
%We define the Sobolev space $(H^{1,p}(X),\|\cdot\|_{H^{1,p}})$ and related $p$-energy functional 
%$\mathscr{E}^{\,p}$ on $H^{1,p}(X)$ based on the quasi-regular strongly local Dirichlet space 
%$(\mathscr{E},D(\mathscr{E}))$ on $L^2(X;\mu)$. 
Let $(\mathscr{E},D(\mathscr{E}))$ be a quasi-regular strongly local Dirichlet form on $L^2(X;\m)$ admitting a minimal 
$\mathscr{E}$-dominant measure $\mu$.  
Set $\mathscr{D}_{1,p}:=\{u\in L^p(X;\m)\cap D(\mathscr{E})\mid \Gamma_{\mu}(u)^{\frac12}\in L^p(X;\mu)\}$ 
and fix a subspace $\mathscr{D}$ of $\mathscr{D}_{1,p}$.   
For $u\in \mathscr{D}$, we define the \emph{$p$-energy} $\mathscr{E}^{\,p}(u):=\|\Gamma_{\mu}(u)^{\frac12}\|_{L^p(X;\mu)}^p$.

\begin{assumption}\label{asmp:closability}
{\rm We assume the closability of $(\mathscr{E}^{\,p},\mathscr{D})$ on $L^p(X;\m)$, i.e. for any 
$\mathscr{E}^{\,p}$-Cauchy sequence $\{u_n\}\subset\mathscr{D}$ 
satisfying $u_n\to0$ in $L^p(X;\m)$ as $n\to\infty$, then 
$\mathscr{E}^{\,p}(u_n)\to0$ as $n\to\infty$. 
}
\end{assumption}
\begin{prop}\label{prop:closability}
Assumption~\ref{asmp:closability} is satisfied if $\m(X)+\mu(X)<\infty$ and $p\geq2$.
\end{prop}
\begin{proof}[\bf Proof]
Suppose that $\{f_n\}\subset \mathscr{D}$ is an $\mathscr{E}^{\,p}$-Cauchy sequence such that 
$f_n\to0$ in $L^p(X;\m)$. Applying H\"older's inequality with $\m(X)+\mu(X)<\infty$ and $p\geq2$, 
we see that 
$\{f_n\}\subset \mathscr{D}$ is an $\mathscr{E}$-Cauchy sequence satisfying $f_n\to0$ in $L^2(X;\m)$. Then $\mathscr{E}(f_n,f_n)\to0$ as $n\to\infty$ by the closedness of $(\mathscr{E},D(\mathscr{E}))$ on 
$L^2(X;\m)$. From this, we see $\Gamma_{\mu}(f_n)\to0$ in $\mu$-measure as $n\to\infty$. 
On the other hand, there exists $v\in L^p(X;\mu)_+$ such that $\Gamma_{\mu}(f_n)^{\frac12}\to v$ in 
$L^p(X;\mu)$ as $n\to\infty$. This implies $v=0$ $\mu$-a.e., hence $\mathscr{E}^{\,p}(f_n)=\int_X\Gamma_{\mu}(f_n)^{\frac{p}{2}}\d\mu\to\int_X v^{\,p}\d\mu=0$ as $n\to\infty$.   
\end{proof}
\begin{remark}
{\rm Proposition~\ref{prop:closability} is also shown in \cite[Theorem~2.1]{HinzKochMeinert} under the condition that 
$\m$ is an $\mathscr{E}$-dominant measure, i.e., $\mathscr{E}$ admits  carr\'e-du-champ $\Gamma$.
}
\end{remark}
\begin{defn}[{\bf {\boldmath$(1,p)$}-Sobolev space {\boldmath$(H^{1,p}(X),\|\cdot\|_{H^{1,p}})$}}]\label{defn:SobolevSpace}
{\rm For $u\in \mathscr{D}$, we define $\|u\|_{H^{1,p}}^p:=\|u\|_{L^p(X;\m)}^p+
\mathscr{E}^{\,p}(u)$. Then $\|\cdot\|_{H^{1,p}}$ forms a norm on $\mathscr{D}$. 
We now define the Sobolev space $H^{1,p}(X)$ as the $\|\cdot\|_{H^{1,p}}$-completion of $\mathscr{D}$: 
\begin{align*}
\left\{\begin{array}{cl}H^{1,p}(X)&:=\overline{\mathscr{D}}^{\|\cdot\|_{H^{1,p}}},\\
\|u\|_{H^{1,p}}&:=\left(\|u\|_{L^p(X;\m)}^p+
\mathscr{E}^{\,p}(u)\right)^{\frac{1}{p}}\quad\text{ for }\quad u\in H^{1,p}(X).\end{array}\right.
\end{align*}
More precisely, for $u\in H^{1,p}(X)$, we can uniquely 
extend  $\Gamma_{\mu}$ to $u$ with $\Gamma_{\mu}(u)^{\frac12}\in L^p(X;\mu)$ 
as the $L^p$-limit of $\Gamma_{\mu}(u_n)^{\frac12}$ for its approximating sequence $\{u_n\}\subset \mathscr{D}$ 
in view of the inequality $|\sqrt{\Gamma_{\mu}(f)}-\sqrt{\Gamma_{\mu}(g)}|\leq\sqrt{\Gamma_{\mu}(f-g)}$ $\mu$-a.e. 
The quantity $\Gamma_{\mu}(u)^{\frac12}$ does not depend on the choice of approximating sequence $\{u_n\}\subset \mathscr{D}$ . 
We set $D(\mathscr{E}^{\,p}):=H^{1,p}(X)$ and  $\mathscr{E}^{\,p}(u):=\|\Gamma_{\mu}(u)^{\frac12}\|_{L^p(X;\mu)}^p$ for $u\in D(\mathscr{E}^{\,p})$, which is called the \emph{$p$-energy} of $u\in H^{1,p}(X)$. 
It is easy to see that for $u,v\in D(\mathscr{E}^{\,p})$, the Parallelogram law (1) above remains valid. 
We can define $\Gamma_{\mu}(u,v)$ as 
\begin{align*}
\Gamma_{\mu}(u,v):=\Gamma_{\mu}\left(\frac{u+v}{2}\right)-\Gamma_{\mu}\left(\frac{u-v}{2}\right)\in L^{\frac{p}{2}}(X;\mu),
\end{align*}
which forms a bilinear form $\Gamma_{\mu}(\cdot,\cdot):H^{1,p}(X)\times H^{1,p}(X)\to L^{\frac{p}{2}}(X;\mu)$ 
under the  Parallelogram law (1) above, 
and Cauchy-Schwarz inequality (3) above also remains valid for $u,v\in D(\mathscr{E}^{\,p})$.

For $u,v\in D(\mathscr{E}^{\,p})$ and $\alpha\geq0$, we define
\begin{align}
\mathscr{E}^{\,p}(u,v):&=\int_X\Gamma_{\mu}(u)^{\frac{p-2}{2}}\Gamma_{\mu}(u,v)\d\mu,\label{eq:CoEnergy}\\
\mathscr{E}^{\,p}_{\alpha}(u,v):&=\mathscr{E}^{\,p}(u,v)+\alpha \int_X|u|^{p-2}uv\,\d\m.\label{eq:PerCoEnergy} 
\end{align}
Clearly we see $\mathscr{E}^{\,p}_0(u,v)=\mathscr{E}^{\,p}(u,v)$ and the linearities of $v\mapsto \mathscr{E}^{\,p}(u,v)$ and $v\mapsto \mathscr{E}^{\,p}_{\alpha}(u,v)$. Moreover, we have 
\begin{align}
|\mathscr{E}^{\,p}(u,v)|&\leq \mathscr{E}^{\,p}(u)^{1-\frac{1}{p}}\mathscr{E}^{\,p}(v)^{\frac{1}{p}},\qquad 
|\mathscr{E}^{\,p}_{\alpha}(u,v)|\leq \mathscr{E}^{\,p}_{\alpha}(u)^{1-\frac{1}{p}}\mathscr{E}^{\,p}_{\alpha}(v)^{\frac{1}{p}}.
\label{eq:ModifiedHolder}
\end{align}
Indeed, the first inequality in \eqref{eq:ModifiedHolder} and the following inequality 
\begin{align}
\left|\alpha\int_X|u|^{p-2}uv\,\d\m\right|\leq\left(\alpha\int_X|u|^p\d\m \right)^{1-\frac{1}{p}}\left(\alpha\int_X|v|^p\d\m  \right)^{\frac{1}{p}}\label{eq:HolderInt}
\end{align}
are easy consequences of H\"older's inequality, and the second inquality in \eqref{eq:ModifiedHolder} follows the first one in \eqref{eq:ModifiedHolder} and \eqref{eq:HolderInt} by way of the usual H\"older's inequality in $\R^2$: $|x_1y_1+x_2y_2|\leq 
(|x_1|^p+|x_2|^p)^{1/p}(|y_1|^q+|y_2|^q)^{1/q}$, $1/p+1/q=1$.    
We call $\mathscr{E}^{\,p}_{\alpha}(u,v)$ the \emph{$\alpha$-order mutual energy of $u$ and $v$}. 
Note that $\mathscr{E}^{\,p}_{\alpha}(u,u)=\int_X\Gamma_{\mu}(u)^{\frac{p}{2}}\d\mu+\alpha\int_X|u|^p\d\m$, in particular, 
$\mathscr{E}^{\,p}_1(u,u)=\mathscr{E}^{\,p}_1(u)=
\|u\|_{H^{1,p}}^p$ and $\mathscr{E}^{\,p}(u,u)=\mathscr{E}^{\,p}(u)$. 
}
\end{defn}
\begin{remark}
{\rm The coenergy \eqref{eq:CoEnergy} and the inequality \eqref{eq:ModifiedHolder} are also noted in \cite{Beznea} under the existence of curr\'e du champ operator. These are considered for the $p$-energy 
associated with $p$-Laplacian 
over Euclidean space (see also \cite{Beusekom}).
}
\end{remark}
\begin{remark}
{\rm To well-define $(1,p)$-Sobolev space $(H^{1,p}(X),\|\cdot\|_{H^{1,p}})$, we need Assumption~\ref{asmp:closability}. Otherwise the unique continuous map $\overline{\iota}: H^{1,p}(X)\to L^p(X;\m)$ 
extending the inclusion map $\iota:\mathscr{D}\to L^p(X;\m)$ is not necessarily injective (see \cite[Chapter I Remark~3.2(ii)]{MR}). 
}
\end{remark}
\begin{remark}
{\rm The normed vector space $(\mathscr{D},\|\cdot\|_{1,p})$ with the following norm: 
\begin{align*}
\|u\|_{1,p}:=\|u\|_{L^p(X;\m)}+\|\Gamma_{\mu}(u)^{\frac12}\|_{L^p(X;\mu)},\quad u\in \mathscr{D}_{1,p}
\end{align*}
is defined in \cite[Notation~6.2.1]{BH} for $p\in[2,+\infty[$ under the condition that $\mu=\m$ is a minimal $\mathscr{E}$-dominant measure, i.e., 
$(\mathscr{E},D(\mathscr{E}))$ admits a carr\'e du champ $\Gamma=\Gamma_{\m}$. But, in \cite{BH}, there is no consideration for its completion. 
}
\end{remark}
\begin{remark}\label{rem:SobolevSpaces}
{\rm 
\begin{enumerate}
\item 
In \cite[Theorem~6.1]{HinzRoecknerTepryaev}, there is a similar construction of $(1,p)$-Sobolev space denoted by $(H^{1,p}_0(X),\|\cdot\|_{1,p})$ under $p>2$ and $\m(X)<\infty$ in the framework of strongly local regular Dirichlet form over locally compact separable metric space. 
The first inequality in \eqref{eq:ModifiedHolder} is also noted in \cite[Remark~6.1]{HinzRoecknerTepryaev}. 
\item There are various ways to define $(1,p)$-Sobolev space for metric measure spaces based on 
curr\'e du champ (see \cite{AGS_Sobolev, BjornBjorn, Gigli:NonSmoothDifferentialStr,GigliNobili}). 
But in the past literature, there is no synthetic treatment for the construction of such spaces based on the theory of strongly local quasi-regular Dirichlet forms admitting carr\'e du champ operator including the case of infinite dimensional spaces, like Wiener space, configuration spaces, interacting particle systems etc. Such an infinite dimensional space has been thought to be considered in the framework of extended metric measure spaces (see also \cite{AES, SavExtendedMetric}).  So the construction of $(1,p)$-Sobolev space based on the theory of strongly local quasi-regular Dirichlet forms is quite applicable. 
\item There is another way to define $(r,p)$-Sobolev spaces (including the case of $r=1$) as the  
Bessel potential space (see \cite{FukKane,HoJacob2004,JacobSchlling2004}). Note that our Sobolev space based on the 
curr\'e du champ operator is equivalent to the $(1,p)$-Sobolev spaces as the Bessel potential space 
 provided the Riesz operator satisfies the $L^p$-boundedness (see \cite[Theorem~1.4(2)]{EXKRiesz}).
\item There is a completely different approach to define $(1,p)$-Sobolev space on fractals (see \cite{HermanPeironeStrichartz, CaoGuQiu, Shimizu, MurganShiumizu}). 
It is unclear the relation with our Sobolev space $(H^{1,p}(X),\|\cdot\|_{H^{1,p}})$. 
\end{enumerate}   
}
\end{remark}

\bigskip

Our first main theorem is the Clarkson type inequality of the $(1,p)$-Sobolev space $(H^{1,p}(X),\|\cdot\|_{H^{1,p}})$. 

\begin{thm}[Clarkson's inequality]\label{thm:Clarkson}
Suppose $p\in]1,+\infty[$. 
The energy functional $\mathscr{E}^{\,p}_1$ satisfies the following inequalities of Clarkson type: for $\alpha\geq0$  and $f,g\in D(\mathscr{E}^{\,p})=H^{1,p}(X)$
\begin{align}
\mathscr{E}^{\,p\!}_{\alpha}\left(\frac{f+g}{2} \right)+\mathscr{E}^{\,p\!}_{\alpha}\left(\frac{f-g}{2} \right)&\leq\frac{\mathscr{E}^{\,p}_{\alpha}(f)+\mathscr{E}^{\,p}_{\alpha}(g)}{2},\quad p\in[2,+\infty[,\label{eq:Clarkson1}\\
\mathscr{E}^{\,p\!}_{\alpha}\left(\frac{f+g}{2} \right)^{\frac{1}{p-1}}+\mathscr{E}^{\,p\!}_{\alpha}\left(\frac{f-g}{2} \right)^{\frac{1}{p-1}}&\leq\left(\frac{\mathscr{E}^{\,p}_{\alpha}(f)+\mathscr{E}^{\,p}_{\alpha}(g)}{2} \right)^{\frac{1}{p-1}},\quad p\in]1,2].\label{eq:Clarkson2}
\end{align}
In particular, $(H^{1,p}(X),\|\cdot\|_{H^{1,p}})$ is a uniformly convex Banach space and consequently it is reflexive. 
\end{thm}

%\begin{remark}
%{\rm 
%
%}
%\end{remark}

%\begin{cor}[Hanner's inequality]\label{cor:Hanner}
%Suppose $p\in]1,+\infty[$. 
%The energy functional $\mathscr{E}^{\,p}_1$ satisfies the following inequalities of Hanner type: for $\alpha\geq0$  and $f,g\in D(\mathscr{E}^{\,p})=H^{1,p}(X)$
%\begin{align}
%\mathscr{E}^{\,p\!}_{\alpha}\left(f+g \right)+\mathscr{E}^{\,p\!}_{\alpha}\left(f-g \right)&\leq
%\left|
%\mathscr{E}^{\,p}_{\alpha}(f)^{\frac{1}{p}}+\mathscr{E}^{\,p}_{\alpha}(g)^{\frac{1}{p}}\right|^p+\left|
%\mathscr{E}^{\,p}_{\alpha}(f)^{\frac{1}{p}}-\mathscr{E}^{\,p}_{\alpha}(g)^{\frac{1}{p}}\right|^p,\;p\in[2,+\infty[,
% \label{eq:Hanner1}\\
%\mathscr{E}^{\,p\!}_{\alpha}\left(f+g \right)+\mathscr{E}^{\,p\!}_{\alpha}\left(f-g \right)&\geq
%\left|
%\mathscr{E}^{\,p}_{\alpha}(f)^{\frac{1}{p}}+\mathscr{E}^{\,p}_{\alpha}(g)^{\frac{1}{p}}\right|^p+\left|
%\mathscr{E}^{\,p}_{\alpha}(f)^{\frac{1}{p}}-\mathscr{E}^{\,p}_{\alpha}(g)^{\frac{1}{p}}\right|^p,\;
%p\in]1,2].\label{eq:Hanner2}
%\end{align}
%\end{cor}

Thanks to the reflexivity of $(H^{1,p}(X),\|\cdot\|_{H^{1,p}})$, we can establish that the
(generalized) normal contraction operates on $(H^{1,p}(X),\|\cdot\|_{H^{1,p}})$. 
For this, we define such notion below. 

\bigskip

For a given function $T:\R^N\to\R$, $T$ is said to be a \emph{(generalized) normal contraction} if 
\begin{align*}
T(0)=0\quad\text{ and }\quad |T(x)-T(y)|\leq \sum_{k=1}^N|x_k-y_k|
\end{align*}
for any $x=(x_1,\cdots,x_N)$, $y=(y_1,\cdots, y_N)\in\R^N$. 
When $N=1$, we call such $T:\R\to\R$ a \emph{normal contraction}. 
Note that $\phi(t):=t^{\sharp}:=0\lor t\land 1$, $\phi(t):=|t|$, $\phi(t):=t^+:=t\lor 0$ and $\phi(t):=t^-:=(-t)\lor0$ are typical examples of normal contractions. More generally, the real-valued function 
$\phi_{\eps}$ on $\R$ depending on $\eps>0$ such that $\phi_{\eps}(t)=t$ for $t\in[0,1]$, $-\eps\leq\phi_{\eps}(t)\leq 1+\eps$ for all $t\in\R$, and $0\leq\phi_{\eps}(t)-\phi_{\eps}(s)\leq t-s$ whenever $s<t$, is also a normal contraction.   

\begin{thm}\label{thm:Noramalcontrtaction1}
Suppose $p\in]1,+\infty[$ and $u=(u_1,\cdots,u_N)\in H^{1,p}(X)^N$. Let  $T:\R^N\to\R$ be a (generalized) normal contraction. Then $T(u)\in H^{1,p}(X)$ and 
$\mathscr{E}^{\,p}(T(u))^{\frac{1}{p}}\leq\sum_{k=1}^N\mathscr{E}^{\,p}(u_k)^{\frac{1}{p}}$, 
hence $\|T(u)\|_{H^{1,p}}\leq 2^{1-\frac{1}{p}}\sum_{k=1}^n\|u_k\|_{H^{1,p}}$. 
\end{thm}
By Theorem~\ref{thm:Noramalcontrtaction1}, we can call $(\mathscr{E}^{\,p},D(\mathscr{E}^{\,p}))$ a \emph{non-linear Dirichlet space} on $L^p(X;\m)$. 

\begin{cor}\label{cor:Noramalcontrtaction}
Suppose $p\in]1,+\infty[$. Then we have the following:
\begin{enumerate}
\item[\rm(1)] 
For any $u\in H^{1,p}(X)$, we have $u^{\sharp},|u|, u^+,u^-\in H^{1,p}(X)$, and 
$\mathscr{E}^{\,p}(u^{\sharp})\leq \mathscr{E}^{\,p}(u)$, $\mathscr{E}^{\,p}(|u|)\leq \mathscr{E}^{\,p}(u)$, 
$\mathscr{E}^{\,p}(u^+)\leq \mathscr{E}^{\,p}(u)$ and $\mathscr{E}^{\,p}(u^-)\leq \mathscr{E}^{\,p}(u)$. 
\item[\rm(2)] For any $u,v\in H^{1,p}(X)$, $u\lor v, u\land v\in H^{1,p}(X)$ and $\mathscr{E}^{\,p}(u\lor v)^{\frac{1}{p}}\leq \mathscr{E}^{\,p}(u)^{\frac{1}{p}}+\mathscr{E}^{\,p}(v)^{\frac{1}{p}}$, $\mathscr{E}^{\,p}(u\land v)^{\frac{1}{p}}\leq \mathscr{E}^{\,p}(u)^{\frac{1}{p}}+\mathscr{E}^{\,p}(v)^{\frac{1}{p}}$. 
\item[\rm(3)] For any $T\in C^1(\R^n)$ with $T(0)=0$, $|\nabla T|\in C_b(\R^N)$ and $u=(u_1,\cdots, u_N)
\in (H^{1,p}(X))^N$, 
we have $T(u)\in H^{1,p}(X)$ with $\mathscr{E}^{\,p}(T(u))^{\frac{1}{p}}\leq \max\limits_{1\leq k\leq N}\left\|\frac{\partial T}{\partial x_k}\right\|_{\infty}\sum_{k=1}^N\mathscr{E}^{\,p}(u_k)^{\frac{1}{p}}$. 
\end{enumerate}
\end{cor}
\begin{thm}\label{thm:Noramalcontrtaction2}
Suppose $p\in]1,+\infty[$ and $u,v\in H^{1,p}(X)$. Then we have the following strong subadditivity of $p$-energy $\mathscr{E}^{\,p}$:
\begin{align}
\mathscr{E}^{\,p}(u\lor v)+\mathscr{E}^{\,p}(u\land v)\leq \mathscr{E}^{\,p}(u)+\mathscr{E}^{\,p}(v).\label{eq:StrongSubAdditive}
\end{align}
\end{thm}

\begin{defn}\label{def:capacity}
{\rm Denote by $\mathcal{O}(X)$ the family of all open subsets of $X$. For $A\in\mathcal{O}(X)$ we define
\begin{align}
\mathcal{L}_A:&=\{u\in D(\mathscr{E}^{\,p})\mid u\geq1 \;\m\text{-a.e.~on }A\}, \label{eq:LA}\\
{\rm Cap}_{1,p}(A):&=\left\{\begin{array}{cc}\inf_{u\in\mathcal{L}_A}\mathscr{E}^{\,p}_1(u,u), & \mathcal{L}_A\ne\emptyset  \\+\infty & 
\mathcal{L}_A=\emptyset,
\end{array}\right.\label{eq:CapacityOpen}
\end{align}
and for any set $A\subset X$ we let 
\begin{align}
{\rm Cap}_{1,p}(A):=\inf_{B\in\mathcal{O}(X),A\subset B}{\rm Cap}_{1,p}(B). \label{eq:CapacityGeneral}
\end{align}
We call this \emph{$1$-order $(1,p)$-capacity} of $A$ or simply the \emph{capacity} of $A$. 
}
\end{defn}

Let $\mathcal{O}_0(X):=\{A\in \mathcal{O}(X)\mid \mathcal{L}_A\ne\emptyset\}$.  
\begin{thm}\label{thm:potential}
For $A,B\in\mathcal{O}_0(X)$, we have the following: 
\begin{enumerate}
\item[{\rm (1)}] There exists a unique element $e_A\in \mathcal{L}_A$ such that 
\begin{align}
\mathscr{E}^{\,p}_1(e_A,e_A)={\rm Cap}_{1,p}(A).\label{eq:potential}
\end{align} 
\item[{\rm (2)}] $0\leq e_A\leq 1$ $\m$-a.e. and $e_A=1$ $\m$-a.e.~on $A$. 
\item[{\rm (3)}] $e_A$ is a unique element of $D(\mathscr{E}^{\,p})$ satisfying $e_A=1$ $\m$-a.e.~on $A$ and 
$\mathscr{E}^{\,p}_1(e_A,v)\geq0$ for all $v\in  D(\mathscr{E}^{\,p})$ with $v\geq0$ $\m$-a.e.~on $A$. 
\item[{\rm (4)}] If $v\in D(\mathscr{E}^{\,p})$ satisfies $v=1$ $\m$-a.e.~on $A$, then $\mathscr{E}^{\,p}_1(e_A,v)={\rm Cap}_{1,p}(A)$. 
\item[{\rm (5)}] If $A\subset B$, then $e_A\leq e_B$ $\m$-a.e.
\end{enumerate}
\end{thm}

\begin{remark}\label{rem:BiroliVernole2}
{\rm 
\begin{enumerate}
\item The assertion of Theorem~\ref{thm:potential}(1),(5) is proved in \cite[Proposition~2.1]{BV2005} 
based on the existence of lower semi continuous convex functional functional $\Phi$ on $L^p(X;\m)$
in the framework of locally compact separable metric space $X$ having positive Radon measure $\m$ with ${\rm supp}[\m]=X$ under several assumptions related to the existence of uniformly convex Banach space in terms of $\Phi$ and the conclusion of Theorem~\ref{thm:Noramalcontrtaction2}. 
In \cite[Proposition~2.1]{BV2005}, they did not show Theorem~\ref{thm:potential}(2) under the condition $(H_6)$ in \cite{BV2005}.
The contents of Theorem~\ref{thm:potential}(3),(4) are not shown in 
\cite[Proposition~2.2]{BV2005}. Theorem~\ref{thm:potential}(1),(5) is not completely covered by 
\cite[Proposition~2.1]{BV2005}, because our underlying space $X$ is not necessarily locally compact.  
\item The assertion of Theorem~\ref{thm:potential}(1) is also proved in \cite[Theorem~2.1]{FukUe2004} in the framework of 
$(1,p)$-Sobolev space of pure jump type by proving Hanner type inequality for Sobolev spaces. 
\item Theorem~\ref{thm:potential}(1) is also proved by 
\cite[Lemma~2.6 and Proposition~2.16]{Hino:ASPM2004} in more general framework, but 
the latter assertion of Theorem~\ref{thm:potential}(2) and  Theorem~\ref{thm:potential}(3),(4),(5) are not shown in \cite{Hino:ASPM2004}. 
\end{enumerate}
}
\end{remark}

\begin{cor}\label{cor:Capacity}
The $(1,p)$-capacity ${\rm Cap}_{1,p}$ is a Choquet capacity, i.e., 
\begin{enumerate}
\item[{\rm (1)}] If $A\subset B$, then ${\rm Cap}_{1,p}(A)\leq{\rm Cap}_{1,p}(B)$. 
\item[{\rm (2)}] If $\{A_n\}_{n\in\mathbb{N}}$ is an increasing sequence of subsets of $X$, then 
\begin{align}
{\rm Cap}_{1,p}\left(\bigcup_{n=1}^{\infty}A_n\right)\leq\sup_{n\in\mathbb{N}}{\rm Cap}_{1,p}(A_n).\label{eq:increasing}
\end{align} 
\item[{\rm (3)}] If $\{A_n\}_{n\in\mathbb{N}}$ is a decreasing sequence of compact subsets of $X$, then 
\begin{align}
{\rm Cap}_{1,p}\left(\bigcap_{n=1}^{\infty}A_n\right)\leq\inf_{n\in\mathbb{N}}{\rm Cap}_{1,p}(A_n).\label{eq:decreasing}
\end{align}
\end{enumerate}
Moreover, it holds that for any Borel set $A$
\begin{align}
{\rm Cap}_{1,p}(A)=\sup_{K:\text{compact}, K\subset A}{\rm Cap}_{1,p}(K),\label{eq:capacitable}
\end{align}
and for any $A,B\subset X$
\begin{align}
{\rm Cap}_{1,p}(A\cup B)+{\rm Cap}_{1,p}(A\cap B)\leq {\rm Cap}_{1,p}(A)+{\rm Cap}_{1,p}(B). \label{eq:strongSubadditiveCap}
\end{align}
\end{cor}
\begin{remark}\label{rem:BiroliVernole1}
{\rm 
\begin{enumerate}
\item The assertion of Corollary~\ref{cor:Capacity} is proved by \cite[Proposition~2.2]{BV2005} 
in their framework noted in Remark~\ref{rem:BiroliVernole2}.
Theorem~\ref{cor:Capacity} is not completely covered by 
\cite[Proposition~2.2]{BV2005}, because our underlying space $X$ is not necessarily locally compact.  
\item The assertion of Corollary~\ref{cor:Capacity} is also proved in \cite[Corollary~2.2]{FukUe2004} in the framework of 
$(1,p)$-Sobolev space of pure jump type. 
\item Except~\eqref{eq:strongSubadditiveCap}, Corollary~\ref{cor:Capacity} is also proved by \cite[Proposition~2.17]{Hino:ASPM2004} in more general context.
\end{enumerate}
}
\end{remark}
\begin{remark}
{\rm For $\alpha>0$, we can define the $\alpha$-order $(1,p)$-capacity ${\rm Cap}_{\alpha,p}(A)$ of $A$ by using $\mathscr{E}^{\,p}_{\alpha}(u,v)$. Similar statements of Theorem~\ref{cor:Capacity}, Theorem~\ref{thm:potential} and Lemma~\ref{lem:Monotone} below 
remain valid for ${\rm Cap}_{\alpha,p}$. Moreover, for $0<\alpha<\beta$, we have 
$\frac{\alpha}{\beta}{\rm Cap}_{\beta,p}(A)\leq{\rm Cap}_{\alpha,p}(A)\leq{\rm Cap}_{\beta,p}(A)$. 
}
\end{remark}

\section{Proofs of Theorems~\ref{thm:Clarkson}, \ref{thm:Noramalcontrtaction1}, Corollary~\ref{cor:Noramalcontrtaction}, Theorem~\ref{thm:Noramalcontrtaction2} and Corollary~\ref{cor:Capacity}}\label{sec:SobolevSpaceH1p}

\begin{proof}[{\bf Proof of Theorem~\ref{thm:Clarkson}}]
We first prove the assertion for $\alpha=0$. 
% is essentially proved in \cite[Proposition~4.4]{GigliNobili} 
%under the framework of infinitesimal Hilbertian metric measure space with the $p$-independence of weak upper gradients in the weak sense 
%(see \cite[Definition~3.1]{GigliNobili}). But its proof remains valid for any fixed $p$. 
For $u,v\in D(\mathscr{E})$, we know 
\begin{align}
\Gamma_{\mu}\left(\frac{u+v}{2}\right)+\Gamma_{\mu}\left(\frac{u-v}{2}\right)=\frac{\Gamma_{\mu}(u)+\Gamma_{\mu}(v)}{2}\quad \mu\text{-a.e.~on }X.\label{eq:curre-DuChampIdentity}
\end{align}
Since any $u\in D(\mathscr{E}^{\,p})$ can be approximated by a sequence $\{u_n\}\subset \mathscr{D}_{1,p}$ in $H^{1,p}$-norm, 
we see the identity \eqref{eq:curre-DuChampIdentity} for $u,v\in D(\mathscr{E}^{\,p})$. 
Suppose $p\in[2,+\infty[$. From the inequality $\|x\|_p\leq\|x\|_2$ for $x\in\R^2$ and \eqref{eq:curre-DuChampIdentity}, 
\begin{align}
\Gamma_{\mu}\left(\frac{u+v}{2}\right)^{\frac{p}{2}}+\Gamma_{\mu}\left(\frac{u-v}{2}\right)^{\frac{p}{2}}&\leq  
\left(\Gamma_{\mu}\left(\frac{u+v}{2} \right)+\Gamma_{\mu}\left(\frac{u-v}{2} \right) \right)^{\frac{p}{2}}\notag\\
&\hspace{-0.1cm}\stackrel{\eqref{eq:curre-DuChampIdentity}}{=}\left(\frac{\Gamma_{\mu}(u)+\Gamma_{\mu}(v)}{2} \right)^{\frac{p}{2}}\notag\\
&\leq \frac{\Gamma_{\mu}(u)^{\frac{p}{2}}+\Gamma_{\mu}(v)^{\frac{p}{2}}}{2},\qquad \mu\text{-a.e.}\label{eq:Clarkson*}
\end{align}
where we use the fact that $\R_{+}\ni t\mapsto t^{\frac{p}{2}}$ is convex in the last step. Integrating this we deduce 
\eqref{eq:Clarkson1} under $\alpha=0$. Next we suppose $p\in]1,2]$. 
Obviously $p\leq2\leq q:=p/(p-1)$ and thus for any $u,v\in D(\mathscr{E}^{\,p})$ we have 
\begin{align}
\Gamma_{\mu}\left(\frac{u+v}{2}\right)^{\frac{q}{2}}+\Gamma_{\mu}\left(\frac{u-v}{2}\right)^{\frac{q}{2}}&
\hspace{-0.1cm}\stackrel{\eqref{eq:Clarkson*}}{\leq} 
\frac{\Gamma_{\mu}(u)^{\frac{q}{2}}+\Gamma_{\mu}(v)^{\frac{q}{2}}}{2}\notag\\
&\leq\left(\frac{\Gamma_{\mu}(u)^{\frac{p}{2}}+\Gamma_{\mu}(v)^{\frac{p}{2}}}{2} \right)^{\frac{q}{p}}, \quad \mu\text{-a.e.}\label{eq:ReverseClarkson}
\end{align}
where in the last inequality we used the fact that $\|x\|_2\leq\|x\|_p$ for any $x\in\R^2$. Thanks to the reverse triangle inequality in $L^r(X;\mu)$ for $r\in]0,1[$ proved in \cite[Proposition~4.4]{GigliNobili}, for any non-negative Borel functions $f,g$, it holds
\begin{align}
\left(\int_X f^r\d\mu\right)^{\frac{1}{r}}+\left(\int_X g^r\d\mu\right)^{\frac{1}{r}}\leq
\left(\int_X (f+g)^r\d\mu\right)^{\frac{1}{r}}.\label{eq:ReverseTriangle}
\end{align}
We apply this with $r:=\frac{p}{q}$, $f:=\Gamma_{\mu}\left(\frac{u+v}{2}\right)^{\frac{q}{2}}$ and 
$g:=\Gamma_{\mu}\left(\frac{u-v}{2}\right)^{\frac{q}{2}}$ to obtain 
\begin{align*}
&\left(\left(\int_X \Gamma_{\mu}\left(\frac{u+v}{2}\right)^{\frac{p}{2}}\d\mu\right)^{\frac{q}{p}}
+\left(\int_X \Gamma_{\mu}\left(\frac{u-v}{2}\right)^{\frac{p}{2}}\d\mu\right)^{\frac{q}{p}}
\right)^{\frac{p}{q}}\\
&\hspace{4cm}\leq \int_X\left(\Gamma_{\mu}\left(\frac{u+v}{2} \right)^{\frac{q}{2}}+\Gamma_{\mu}\left(\frac{u-v}{2} \right)^{\frac{q}{2}} \right)^{\frac{p}{q}}\d\mu.
\end{align*}
This and \eqref{eq:ReverseClarkson} give \eqref{eq:Clarkson2} under $\alpha=0$. 

Next we prove the case for $\alpha>0$. In this case, the proof for $p\in[2,+\infty[$ is easy, because it is done by adding 
  \eqref{eq:Clarkson1} under $\alpha=0$ to the usual Clarkson's inequality over $L^p(X;\m)$. Next we assume 
  $p\in]1,2]$. We have two inequalities of Clarkson type below: 
\begin{align}
\mathscr{E}^{\,p\!}\left(\frac{f+g}{2} \right)^{\frac{1}{p-1}}+\mathscr{E}^{\,p\!}\left(\frac{f-g}{2} \right)^{\frac{1}{p-1}}&\leq\left(\frac{\mathscr{E}^{\,p}(f)+\mathscr{E}^{\,p}(g)}{2} \right)^{\frac{1}{p-1}},\label{eq:Clarkson3}\\
\left\|\alpha\cdot\frac{f+g}{2}\right\|_{L^p(X;\m)}^{\frac{p}{p-1}}+\left\|\alpha\cdot\frac{f-g}{2}\right\|_{L^p(X;\m)}^{\frac{p}{p-1}}&\leq 
\left(\alpha\cdot\frac{\|f\|_{L^p(X;\m)}^p+\|g\|_{L^p(X;\m)}^p}{2} \right)^{\frac{p}{p-1}}.\label{eq:Clarkson4}
\end{align}
Then we can deduce \eqref{eq:Clarkson2} from \eqref{eq:Clarkson3} and \eqref{eq:Clarkson4} by using the closed convex body  
 $\overline{B}_1^{\,p}(0):=\{(x,y)\in\R^2\mid |x|^{\frac{p}{p-1}}+|y|^{\frac{p}{p-1}}\leq1\}$ under $p\in]1,2]$. 
\end{proof}
%\begin{cor}\label{cor:reflexivity}
%Suppose $p\in]1,+\infty[$. Then 
%$(H^{1,p}(X),\|\cdot\|_{H^{1,p}})$ is a uniformly convex Banach space and consequently it is reflexive. 
%\end{cor}
\begin{remark}
{\rm The content of Theorem~\ref{thm:Clarkson} %and Corollary~\ref{cor:reflexivity} 
is proved in \cite[Proposition~4.4]{GigliNobili} in the framework of infinitesimal Hilbertian metric measure space with $p$-independence of weak upper gradients in the weak sense. 
Our construction of $(1,p)$-Sobolev space is completely different from the Sobolev space $W^{1,p}(X)$ constructed in 
\cite{GigliNobili} and does not need to impose the $p$-independence of weak upper gradient.  
}
\end{remark}
\begin{proof}[{\bf Proof of Theorem~\ref{thm:Noramalcontrtaction1}}]
First we assume $T\in C^{\infty}(\R^N)$. Then the generalized normal contraction property implies $|\partial T/\partial x_k|\leq1$ on $\R^N$ for each $k\in\{1,\cdots,N\}$. By \cite[Corollary~6.1.3]{BH} or \cite[Lemma~5.2]{Kw:func}, under $u\in (\mathscr{D}_{1,p})^N
\subset (D(\mathscr{E}))^N$, we have
\begin{align}
\Gamma_{\mu}(T(u))&=\left|\Gamma_{\mu}(T(u))\right|\notag\\&=\left|
\sum_{i,j=1}^N\frac{\partial T}{\partial x_i}(u)\frac{\partial T}{\partial x_j}(u)\Gamma_{\mu}(u_i,u_j)\right|\notag\\
&\leq \sum_{i,j=1}^N\left|\frac{\partial T}{\partial x_i}(u)\right|\left|\frac{\partial T}{\partial x_j}(u)\right|
\left|\Gamma_{\mu}(u_i,u_j)\right|\notag\\&\leq \sum_{i,j=1}^N|\Gamma_{\mu}(u_i,u_j)|\label{eq:Bouleau}\\&\leq \sum_{i,j=1}^N\Gamma_{\mu}(u_i)^{\frac12}\Gamma_{\mu}(u_j)^{\frac12}\notag\\
&=\left(\sum_{i=1}^N\Gamma_{\mu}(u_i)^{\frac12} \right)^2\quad \mu\text{-a.e.}\notag
\end{align}
From this, we can easily confirm $T(u)\in \mathscr{D}_{1,p}$ for $u=(u_1,\cdots, u_N)\in(\mathscr{D}_{1,p})^N$ and 
$\mathscr{E}^{\,p}(T(u))^{\frac{1}{p}}\leq \sum_{k=1}^N\mathscr{E}^{\,p}(u_k)^{\frac{1}{p}}$.  
Next we 
take $u\in  H^{1,p}(X)^N$. Then there exists an $\|\cdot\|_{1,p}$-approximating sequence $\{u^n\}\subset (\mathscr{D}_{1,p})^N$ to $u$ in the sense that $u_k^n\to u_k$ in $(H^{1,p}(X),\|\cdot\|_{H^{1,p}})$ as  $n\to\infty$ for each $k\in\{1,\cdots, N\}$. 
In particular, $\{u^n\}$ is $L^p(X;\m)$-convergent to $u$ in the sense that $u_k^n\to u_k$ in $L^p(X;\m)$ as  $n\to\infty$ for each $k\in\{1,\cdots, N\}$. 
Moreover, $\{T(u^n)\}$ is also $L^p(X;\m)$-convergent to $T(u)$. Indeed, since $|T(u^n)|\leq\sum_{k=1}^N|u^n_k|$, we see 
that $\{T(u^n)\}$ is uniformly integrable and it converges to $T(u)$ in $\m$-measure. 
This implies the $L^p$-convergence of $\{T(u^n)\}$ to $T(u)$ (see Vitali's convergence theorem in \cite[Theorem~16.6]{RneSchilling}).

It is easy to see that $\{T(u^n)\}$ is a $\|\cdot\|_{H^{1,p}}$-bounded sequence of $\mathscr{D}_{1,p}$. 
In view of the reflexivity of $(H^{1,p}(X),\|\cdot\|_{H^{1,p}})$, we can apply Banach-Alaoglu Theorem to the topological dual space $(H^{1,p}(X),\|\cdot\|_{H^{1,p}})=(H^{1,p}(X)^{**},\|\cdot\|_{H^{1,p}(X)^{**}})$ 
of $(H^{1,p}(X)^*,\|\cdot\|_{H^{1,p}(X)^*})$ 
so that there exists a subsequence $\{T(u^{n_i})\}\subset \mathscr{D}_{1,p}$ and 
$v\in H^{1,p}(X)$ such that $\{T(u^{n_i})\}$ $H^{1,p}(X)$-weakly converges to $v$, in particular, it converges to $v$ $L^p$-weakly, because any $\ell\in L^p(X;\mu)^*=L^q(X;\mu)$ can be regarded as $\ell|_{H^{1,p}(X)^*}\in H^{1,p}(X)^*$. 
Since $\{T(u^n)\}$ converges to $T(u)$ in $L^p(X;\m)$, we see $v=T(u)$.  
From this, we have
\begin{align*}
\|T(u)\|_{L^p(X;\m)}^p+\|\Gamma_{\mu}(T(u))^{\frac12}\|_{L^p(X;\mu)}^p&=
\|T(u)\|_{H^{1,p}}^p\\
&\leq\varliminf_{i\to\infty}\|T(u^{n_i})\|_{H^{1,p}}^p\\
&\leq\varliminf_{i\to\infty}\left(\|T(u^{n_i})\|_{L^p(X;\m)}^p+
\|\Gamma_{\mu}(T(u^{n_i}))^{\frac12}\|_{L^p(X;\mu)}^p
\right)\\
&=\|T(u)\|_{L^p(X;\m)}^p+\varliminf_{i\to\infty}\|\Gamma_{\mu}(T(u^{n_i}))^{\frac12}\|_{L^p(X;\mu)}^p
\\
&\leq\|T(u)\|_{L^p(X;\m)}^p+\varliminf_{i\to\infty}\left(\sum_{k=1}^N\|\Gamma_{\mu}(u^{n_i}_k)^{\frac12}\|_{L^p(X;\mu)}\right)^{p}\\
&=\|T(u)\|_{L^p(X;\m)}^p+\left(\sum_{k=1}^N\|\Gamma_{\mu}(u_k)^{\frac12}\|_{L^p(X;\mu)}\right)^p.
\end{align*}
In the last inequality above, we use the fact that $u^{n_i}\in(\mathscr{D}_{1,p})^N$ implies $T(u^{n_i})\in \mathscr{D}_{1,p}$ and $\Gamma_{\mu}(T(u^{n_i}))^{\frac12}\leq\sum_{k=1}^N\Gamma_{\mu}(u^{n_i}_k)^{\frac12}$ $\mu$-a.e.~on $X$ shown in \eqref{eq:Bouleau}.
Consequently, we obtain that 
\begin{align*}
\mathscr{E}^{\,p}(T(u))^{\frac{1}{p}}=\|\Gamma_{\mu}(T(u))^{\frac12}\|_{L^p(X;\mu)}\leq \sum_{k=1}^N\|\Gamma_{\mu}(u_k)^{\frac12}\|_{L^p(X;\mu)}=\sum_{k=1}^N\mathscr{E}^{\,p}(u_k)^{\frac{1}{p}}
\end{align*}
and 
\begin{align*}
\|T(u)\|_{H^{1,p}(X)}\leq2^{1-\frac{1}{p}}\sum_{k=1}^N\|u_k\|_{H^{1,p}(X)}.
\end{align*} 
Finally, we consider the general case. Since $x_k\mapsto T(x_1,\cdots, x_k,\cdots, x_N)$ is a $1$-Lipschitz function on $\R$, 
it is differential almost everywhere on $\R$ in view of Rademacher Theorem and its $k$-th derivative $\frac{\partial T}{\partial x_k}$ exists for a.e.~$x_k$ and $|\frac{\partial T}{\partial x_k}|\leq 1$ a.e.~$x_k$ for each $k\in\{1,\cdots, x_N\}$. In particular, $x_k\mapsto \frac{\partial T}{\partial x_k}$ is locally integrable on $\R$. 

Take a function $j$ defined by $j(x):=e^{-\frac{1}{1-|x|^2}}$ for $|x|\leq1$ and $j(x):=0$ for $|x|>1$. 
Set $j_n(x):=n^Nj(nx)$ and $T_n(x):=j_n*T(x)-j_n*T(0):=\int_{\R^N}(j_n(x-y)-j_n(y))T(y)\d y$. Then $T_n\in C^{\infty}(\R^N)$, $T_n(0)=0$ and $\lim_{n\to\infty}T_n(x)=T(x)$ for each $x\in\R^N$.
We easily see 
\begin{align}
\frac{\partial T_n}{\partial x_k}(x)&=\int_{\R^N}\frac{\partial }{\partial x_k}j_n(x-y)T(y)\d y\notag 
\\&=-\int_{\R^N}
\frac{\partial }{\partial y_k}j_n(x-y)T(y)\d y
=\int_{\R^N}j_n(x-y)\frac{\partial T}{\partial y_k}(y)\d y,\label{eq:smoothing}
\end{align}
because 
\begin{align*}
-\int_{\R}
\frac{\partial }{\partial y_k}j_n(x-y)T(y)\d y_k&=-\lim_{-a,b\to\infty}\left[j_n(x-y)T(y) \right]_{y_k=a}^{y_k=b}+\int_{\R}j_n(x-y)\frac{\partial T}{\partial y_k}(y)\d y_k\\
&=\int_{\R}j_n(x-y)\frac{\partial T}{\partial y_k}(y)\d y_k.
\end{align*}
From \eqref{eq:smoothing}, we see $|\frac{\partial T_n}{\partial x_k}|\leq1$ on $\R^N$ for all $k\in\{1,\cdots, N\}$ and $T_n(0)=0$, hence $T_n$ is a generalized normal contraction on $\R^N$.  
Now take $u\in H^{1,p}(X)^N$. Then we already know $T_n(u)\in H^{1,p}(X)$ and 
$\|\Gamma_{\mu}(T_n(u))^{\frac12}\|_{L^p(X;\mu)}\leq\sum_{k=1}^N\|\Gamma_{\mu}(u_k)^{\frac12}\|_{L^p(X;\mu)}$, and 
$\|T_n(u)\|_{H^{1,p}}\leq 2^{1-\frac{1}{p}}\sum_{k=1}^N\|u_k\|_{H^{1,p}}$.  

Since $|T_n(u)|\leq \sum_{k=1}^N|u_k|$ and $T_n(u)\to T(u)$ as $n\to\infty$ $\m$-a.e., $\{T_n(u)\}$ $L^p(X;\m)$-converges 
to $T(u)$ as $n\to\infty$. Moreover $\Gamma_{\mu}(T_n(u))^{\frac12}\leq \sum_{k=1}^N\Gamma_{\mu}(u_k)^{\frac12}$ $\mu$-a.e. yields the 
$\|\cdot\|_{H^{1,p}}$-boundedness of $\{T_n(u)\}$. In view of the reflexivity of $(H^{1,p}(X),\|\cdot\|_{H^{1,p}})$, 
there exists an $H^{1,p}$-weakly convergent subsequence $\{T_{n_i}(u)\}$ to $T(u)$. Therefore, we can deduce 
\begin{align*}
\|T(u)\|_{L^p(X;\m)}^p+\|\Gamma_{\mu}(T(u))^{\frac12}\|_{L^p(X;\mu)}^p&=
\|T(u)\|_{H^{1,p}}^p\\
&\leq\varliminf_{i\to\infty}\|T_{n_i}(u)\|_{H^{1,p}}^p\\
&\leq\varliminf_{i\to\infty}\left(\|T_{n_i}(u)\|_{L^p(X;\m)}^p+
\|\Gamma_{\mu}(T_{n_i}(u))^{\frac12}\|_{L^p(X;\mu)}^p
\right)\\
&=\|T(u)\|_{L^p(X;\m)}^p+\varliminf_{i\to\infty}\|\Gamma_{\mu}(T_{n_i}(u))^{\frac12}\|_{L^p(X;\mu)}^p
\\
&\leq\|T(u)\|_{L^p(X;\m)}^p+\left(\sum_{k=1}^N\|\Gamma_{\mu}(u_k)^{\frac12}\|_{L^p(X;\mu)}\right)^{p}.
\end{align*}
Therefore, we obtain 
\begin{align*}
\mathscr{E}^{\,p}(T(u))^{\frac{1}{p}}=\|\Gamma_{\mu}(T(u))^{\frac12}\|_{L^p(X;\mu)}\leq \sum_{k=1}^N\|\Gamma_{\mu}(u_k)^{\frac12}\|_{L^p(X;\mu)}= \sum_{k=1}^N\mathscr{E}^{\,p}(u_k)^{\frac{1}{p}}
\end{align*} 
and 
\begin{align*}
\|T(u)\|_{H^{1,p}}\leq 2^{1-\frac{1}{p}}\sum_{k=1}^N\|u_k\|_{H^{1,p}}. 
\end{align*}
\end{proof}

\begin{proof}[{\bf Proof of Corollary~\ref{cor:Noramalcontrtaction}}]
The assertion is an easy consequence of Theorem~~\ref{thm:Noramalcontrtaction1}. We omit it. 
\end{proof}

\begin{proof}[{\bf Proof of Theorem~\ref{thm:Noramalcontrtaction2}}]
First we prove the assertion for $u,v\in \mathscr{D}_{1,p}$. In this case, by way of truncation property of carr\'e du champ 
$\Gamma_{\mu}$, we have
\begin{align*}
\Gamma_{\mu}(u\lor v)^{\frac{p}{2}}+\Gamma_{\mu}(u\land v)^{\frac{p}{2}}&=\1_{\{u\geq v\}}\Gamma_{\mu}(u)^{\frac{p}{2}}+\1_{\{u< v\}}\Gamma_{\mu}(v)^{\frac{p}{2}}
+ \1_{\{u\geq v\}}\Gamma_{\mu}(v)^{\frac{p}{2}}+\1_{\{u< v\}}\Gamma_{\mu}(u)^{\frac{p}{2}}\\
&=\Gamma_{\mu}(u)^{\frac{p}{2}}+\Gamma_{\mu}(v)^{\frac{p}{2}},\quad \mu\text{-a.e.}
\end{align*}
Integrating this, we obtain the conclusion. Next, we consider general $u,v\in H^{1,p}(X)$. 
Then there exist sequences $\{u_n\},\{v_n\}\subset \mathscr{D}_{1,p}$ converging to $u,v$ in $H^{1,p}(X)$, respectively. 
By Corollary~\ref{cor:Noramalcontrtaction}(2), $\{u_n\lor v_n\}$ and $\{u_n\land v_n\}$ are $\|\cdot\|_{H^{1,p}}$-bounded. 
Then, by way of the reflexivity of $(H^{1,p}(X),\|\cdot\|_{H^{1,p}})$,  
there exists a common subsequence $\{n_i\}$ such that $\{u_{n_i}\lor v_{n_i}\}$ and $\{u_{n_i}\land v_{n_i}\}$ $H^{1,p}$-weakly converges to  
$u\lor v$ and $u\land v$, respectively. Therefore, 
\begin{align*}
\mathscr{E}^{\,p}(u\lor v)+\mathscr{E}^{\,p}(u\land v)
&=\|\Gamma_{\mu}(u\lor v)^{\frac12}\|_{L^p(X;\mu)}^p+\|\Gamma_{\mu}(u\land v)^{\frac12}\|_{L^p(X;\mu)}^p
\\
&\leq \varliminf_{i\to\infty}\|\Gamma_{\mu}(u_{n_i}\lor v_{n_i})^{\frac12}\|_{L^p(X;\mu)}^p
+\varliminf_{i\to\infty}\|\Gamma_{\mu}(u_{n_i}\land v_{n_i})^{\frac12}\|_{L^p(X;\mu)}^p
\\
&\leq \varliminf_{i\to\infty}\left(\|\Gamma_{\mu}(u_{n_i}\lor v_{n_i})^{\frac12}\|_{L^p(X;\mu)}^p +\|\Gamma_{\mu}(u_{n_i}\land v_{n_i})^{\frac12}\|_{L^p(X;\mu)}^p\right) \\
&=\varliminf_{i\to\infty}\left(\|\Gamma_{\mu}(u_{n_i})^{\frac12}\|_{L^p(X;\mu)}^p +\|\Gamma_{\mu}(v_{n_i})^{\frac12}\|_{L^p(X;\mu)}^p\right)\\
&=\|\Gamma_{\mu}(u)^{\frac12}\|_{L^p(X;\mu)}^p +\|\Gamma_{\mu}(v)^{\frac12}\|_{L^p(X;\mu)}^p\\
&=\mathscr{E}^{\,p}(u)+\mathscr{E}^{\,p}(v).
\end{align*}
\end{proof}

\begin{proof}[{\bf Proof of Theorem~\ref{thm:potential}}]
First we prove (1) under $p\in[2,+\infty[$. The proof under $p\in]1,2]$ is similar. 
Let $\{u_n\}\subset \mathcal{L}_A$ be a minimizing sequence to ${\rm Cap}_{1,p}(A)=\inf_{u\in\mathcal{L}_A}\mathcal{E}^{\,p}_1(u_n)$. By \eqref{eq:Clarkson1}, we can see that any minimizing sequence $\{u_n\}\subset \mathcal{L}_A$ ($\lim_{n\to\infty}\mathscr{E}^{\,p}_1(u_n)={\rm Cap}_{1,p}(A)$) is $H^{1,p}$-convergent to an element $e_A\in\mathcal{L}_A$ satisfying \eqref{eq:potential} and that such an $e_A$ is unique. 

(2) By Theorem~\ref{thm:Noramalcontrtaction1}, $u:=(0\lor e_A)\land 1\in\mathcal{L}_A$ and $\mathscr{E}^{\,p}_1(u,u)\leq\mathscr{E}^{\,p}_1(e_A,e_A)={\rm Cap}_{1,p}(A)$. Hence $u=e_A$. 

(3) If $v$ has the stated property, then $e_A+\eps v\in\mathcal{L}_A$ and $\mathscr{E}^{\,p}_1(e_A+\eps v, e_A+\eps v)\geq\mathscr{E}^{\,p}_1(e_A,e_A)$ for any $\eps\geq0$. This means that 
\begin{align*}
0&\leq \left.\frac{\d}{\d\eps}\mathscr{E}^{\,p}_1(e_A+\eps v, e_A+\eps v)\right|_{\eps=0}=p\,\mathscr{E}^{\,p}_1(e_A,v)\quad \text{ for }v\in D(\mathscr{E}^{\,p})\text{ with }v\geq0\;\m\text{-a.e.~on }A.
\end{align*}
Indeed, 
\begin{align*}
\left.\frac{\d}{\d\eps}\mathscr{E}^{\,p}_1(e_A+\eps v, e_A+\eps v)\right|_{\eps=0}&=\left.\frac{p}{2}\int_X\Gamma_{\mu}(e_A+\eps v)^{\frac{p-2}{2}}\left(2\eps\Gamma_{\mu}(v,v)+2\Gamma_{\mu}(e_A,v) \right)\d\mu\right|_{\eps=0}\\
&\hspace{1cm}+\left.\frac{p}{2}\int_X|e_A+\eps v|^{p-2}(2\eps|v|^2+2e_Av)\d\m\right|_{\eps=0}\\
&=p\int_X\Gamma_{\mu}(e_A)^{\frac{p}{2}-1}\Gamma_{\mu}(e_A,v)\d\mu+p\int_Xe_A^{\,p-1}v\d\m\\
&=p\,\mathcal{E}^{\,p}_1(e_A,v).
\end{align*}
Convsersely, suppose $u\in D(\mathscr{E}^{\,p})$ satisfies the conditions in (3), then $u\in\mathcal{L}_A$ and $w-u\geq0$ $\m$-a.e.~on $A$ for any $w\in\mathcal{L}_A$. In particular, $\mathscr{E}^{\,p}_1(u,w-u)\geq0$. Hence 
\begin{align*}
\mathscr{E}^{\,p}_1(u,u)\leq\mathscr{E}^{\,p}_1(u,w)\leq \mathscr{E}^{\,p}_1(u,u)^{1-\frac{1}{p}}
\mathscr{E}^{\,p}_1(w,w)^{\frac{1}{p}},
\end{align*}
which implies $\mathscr{E}^{\,p}_1(u,u)\leq\mathscr{E}^{\,p}_1(w,w)$, proving $u=e_A$. 

(4) follows immediately from (3). 

(5) Since $A\subset B$, $e_A=e_A\land e_B$ $\m$-a.e.~on $A$. This implies $\mathscr{E}^{\,p}_1(e_A,e_A)=\mathscr{E}^{\,p}_1(e_A,e_A\land e_B)$, because $\mathscr{E}^{\,p}_1(e_A,v)=0$ if $v\in D(\mathscr{E}^{\,p})$ satisfies $v=0$ $\m$-a.e.~on $A$. 
Therefore, 
\begin{align*}
\mathscr{E}^{\,p}_1(e_A,e_A)&=\mathscr{E}^{\,p}_1(e_A,e_A\land e_B)\\
&\leq \mathscr{E}^{\,p}_1(e_A,e_A)^{1-\frac{1}{p}}\mathscr{E}^{\,p}_1(e_A\land e_B,e_A\land e_B)^{\frac{1}{p}},
\end{align*} 
which implies $\mathscr{E}^{\,p}_1(e_A,e_A)\leq \mathscr{E}^{\,p}_1(e_A\land e_B,e_A\land e_B)$, proving $e_A=e_A\land e_B$. 
\end{proof}

To prove Corollary~\ref{cor:Capacity}, we need the following lemma.

\begin{lem}\label{lem:Monotone}
We have the following: 
\begin{enumerate}
\item[{\rm (1)}] If $A,B\in\mathcal{O}(X)$ with $A\subset B$, then ${\rm Cap}_{1,p}(A)\leq{\rm Cap}_{1,p}(B)$. 
\item[{\rm (2)}] \eqref{eq:strongSubadditiveCap} holds for $A,B\in \mathcal{O}(X)$. 
\item[{\rm (3)}] Let $\{A_n\}$ be an increasing sequence of $A_n\in\mathcal{O}(X)$. Then \eqref{eq:increasing} holds. 
\end{enumerate}
\end{lem}
\begin{proof}[{\bf Proof}]
(1) is clear. (2) For $A,B\in\mathcal{O}_0(X)$, by Theorem~\ref{thm:Noramalcontrtaction2}, 
\begin{align*}
{\rm Cap}_{1,p}(A\cup B)+{\rm Cap}_{1,p}(A\cap B)&\leq \mathscr{E}^{\,p}_1(e_A\lor e_B,e_A\lor e_B)+\mathscr{E}^{\,p}_1(e_A\land  e_B,e_A\lor e_B)\\
&\leq \mathscr{E}^{\,p}_1(e_A,e_A)+\mathscr{E}^{\,p}_1(e_B,e_B)\\
&={\rm Cap}_{1,p}(A)+{\rm Cap}_{1,p}(B).
\end{align*}
(3) Suppose $\sup_{n\in\mathbb{N}}{\rm Cap}_{1,p}(A_n)<\infty$. We prove the assertion for the case $p\in[2,+\infty[$. 
The proof for the case $p\in]1,2]$ is similar based on \eqref{eq:Clarkson2}. 
Suppose $n>m$. Then $\frac{e_{A_n}+e_{A_m}}{2}\in\mathcal{L}_{A_m}$. 
According to the inequality \eqref{eq:Clarkson1} of Clarkson type,  
\begin{align*}
\mathscr{E}^{\,p}_1\left(\frac{e_{A_n}-e_{A_m}}{2} \right)+{\rm Cap}_{1,p}(A_m)&\leq 
\mathscr{E}^{\,p}_1\left(\frac{e_{A_n}-e_{A_m}}{2} \right)+\mathscr{E}^{\,p}_1\left(\frac{e_{A_n}+e_{A_m}}{2} \right)\\
&\leq 
\frac{\mathscr{E}^{\,p}_1(e_{A_n})+\mathscr{E}^{\,p}_1(e_{A_m})}{2}\\
&=\frac12\left({\rm Cap}_{1,p}(A_n)+{\rm Cap}_{1,p}(A_m) \right),
\end{align*}
which implies that $\{e_{A_n}\}_{n\in\mathbb{N}}$ converges to some $u\in D(\mathscr{E}^{\,p})$ in $H^{1,p}$-norm. 
Clearly $u=1$ 
$\m$-a.e.~on $A:=\bigcup_{n=1}^{\infty}A_n$. If $v\in D(\mathscr{E}^{\,p})$ is non-negative $\m$-a.e.~on $A$, then 
$\mathscr{E}^{\,p}_1(u,v)=\lim_{n\to\infty}\mathscr{E}^{\,p}_1(e_{A_n},v)\geq0$. Hence $u=e_A$ by Lemma~\ref{thm:potential}(3) 
and $\sup_{n\in\mathbb{N}}{\rm Cap}_{1,p}(A_n)=\lim_{n\to\infty}\mathscr{E}^{\,p}_1(e_{A_n},e_{A_n})=\mathscr{E}^{\,p}_1(u,u)={\rm Cap}_{1,p}(A)$. 
\end{proof}
\begin{proof}[{\bf Proof of Corollary~\ref{cor:Capacity}}]
The conclusion follows Lemma~\ref{lem:Monotone} and \cite[Theorem~A.1.2]{FOT}. 
\eqref{eq:capacitable} is a consequence of \cite[Theorem~A.1.1]{FOT}. 
\eqref{eq:strongSubadditiveCap} directly follows from Lemma~\ref{lem:Monotone}(3). 
\end{proof}
\section{Examples}
\begin{example}[Abstract Wiener space]
{\rm 
Let  $(B,H,\mu)$ be an abstract Wiener space, i.e. $B$ is a real separable Banach space and $H$ is a real separable
Hilbert space that is densely and continuously imbedded in $B$ and $\mu$ is a Gaussian measure on $B$ satisfying 
\begin{align}
\int_B\exp\left(\sqrt{-1}{}_B\langle w,\varphi\rangle_{B^*} \right)\d\mu=\exp\left(-\frac12|\varphi|_{H^*}^2\right)\quad \text{ for }\quad \varphi\in B^*\subset H^*.
\end{align}
(see Shigekawa~\cite[Definition~1.1]{ShigekawaText}). Let $(T_t^{\rm OU})_{t\geq0}$ be the Ornstein-Uhlenbeck semigroup defined by 
\begin{align}
T_t^{\rm OU}f(x):=\int_Bf(e^{-t}x+\sqrt{1-e^{-2t}}y)\mu(\d y)\quad \text{ for }\quad f\in\mathscr{B}(B)_+.
\end{align}
It is known that $T_t^{\rm OU}$ is $\mu$-symmetric and 
$(T_t^{\rm OU})_{t\geq0}$ can be extended to a strongly continuous contraction semigroup in $L^p(B;\mu)$ for $p\geq1$ (see  
Shigekawa~\cite[Proposition~2.4]{ShigekawaText}). 
Let $(\mathscr{E}^{\rm OU},D(\mathscr{E}^{\rm OU}))$ be the Dirichlet form on 
$L^2(B;\mu)$ associated to  $(T_t^{\rm OU})_{t\geq0}$. It is known that $(\mathscr{E}^{\rm OU},D(\mathscr{E}^{\rm OU}))$ is associated to a diffusion process ${\bf X}^{\rm OU}$ called \emph{Ornstein-Uhlenbeck process}. 
A function $f:B\to\R$ is called $H$-differentiable at $x\in B$ if there exists an $h^*\in H^*$ such that 
$\left.\frac{\d}{\d t}f(x+th)\right|_{t=0}={}_H\langle h,h^*\rangle_{H^*}$ $h\in H$. Then $h^*$ is called the $H$-derivative of $f$ at $x$ denoted by $D_Hf(x)$. 
Let 
\begin{align*}
\mathcal{P}:&=\{F:B\to\R\mid \text{ there exists }n\in\N, \varphi_1,\cdots,\varphi_n\in B^*, \\
&\hspace{3cm} f\in C^{\infty}(\R^n) \text{ is a polynomial on }\R^n\text{ such that }\\
&\hspace{7cm} F(x)=f(\langle x,\varphi_1\rangle,\cdots,\langle x,\varphi_n\rangle)\}
\end{align*}
be the family of polynomial functions on $B$. For $f\in\mathcal{P}$, its $H^{1,p}$-norm $\|f\|_{H^{1,p}}$ is given by 
\begin{align*}
\|f\|_{H^{1,p}}^p:=\int_B|f|^p\d\mu+\int_B|D_Hf|_{H^*}^p\d\mu<\infty.
\end{align*}
When $p=2$, the $H^{1,2}$-completion of $(\mathcal{P},\|\cdot\|_{H^{1,2}})$ is nothing but the Dirichlet form 
$(\mathscr{E}^{\rm OU},D(\mathscr{E}^{\rm OU}))$ and it admits carr\'e du champ $\Gamma$ given by 
$\Gamma(f)=|D_Hf|_{H^*}^2$, where $D_H:D(\mathscr{E}^{\rm OU})\to L^2(B,H^*;\mu)$ is a natural extension of 
$D_H$. It is  shown in \cite[Theorem~4.4]{ShigekawaText} that the norm $\|\cdot\|_{H^{1,p}}$ on $\mathcal{P}$ is equivalent to another 
norm $\|\cdot\|_{W^{1,p}}$ on $\mathcal{P}$ based on the Ornstein-Uhlenbeck operator $L^{\rm OU}$ 
associated to $(T_t^{\rm OU})_{t\geq0}$. In particular, our Assumption~\ref{asmp:closability} for $\mathscr{D}=\mathcal{P}$ is satisfied. 
Then our $(1,p)$-Sobolev space $(H^{1,p}(B),\|\cdot\|_{H^{1,p}})$ coincides with 
the $H^{1,p}$-completion of $\mathcal{P}(\subset L^p(B;\mu))$. 
In the case of abstract Wiener space, thanks to \cite{Day, RoecknerSchmuland}, 
%the 
%equivalence of Sobolev norms between $(H^{1,p}(B),\|\cdot\|_{H^{1,p}})$ and Bessel potential space (see Remark~\ref{rem:SobolevSpaces}(iii)),  
it is easy to see the uniform convexity of $(H^{1,p}(B),\|\cdot\|_{H^{1,p}})$ without using Theorem~\ref{thm:Clarkson}. However, the contents of Theorem~\ref{thm:potential}(3),(4),(5) are not known for abstract Wiener space and the $(1,p)$-Sobolev space on it. 
}
\end{example}

\begin{example}[{\bf {\boldmath${\sf RCD}(K,\infty)$}-space}]
{\rm A metric measure space is a triple $(X,{\sf d},\m)$ such that 
\begin{align*}
(X,{\sf d}):\quad &\text{ is a complete separable metric space},\\
\m\ne0:\quad&\text{ is non-negative and boundedly finite Borel measure}.
\end{align*}
Any metric open ball is denoted by $B_r(x):=\{y\in X\mid {\sf d}(x,y)<r\}$ for $r>0$ and $x\in X$. 
A subset $B$ of $X$ is said to be bounded if it is included in a metric open ball.    
Denote by $C([0,1],X)$ the space of continuous curve defined on the unit interval $[0,1]$ equipped the distance ${\sf d}_{\infty}(\gamma,\eta):=\sup_{t\in[0,1]}{\sf d}(\gamma_t,\eta_t)$ for every $\gamma,\eta\in C([0,1],X)$. This turn $C([0,1],X)$ into complete separable metric space.   
Next we consider the set of $q$-absolutely continuous curves, denoted by $AC^q([0,1],X)$, is the subset of $\gamma\in C([0,1],X)$ so that there exists  $g\in L^q(0,1)$ satisfying 
\begin{align*}
{\sf d}(\gamma_t,\gamma_s)\leq\int_s^tg(r)\d r,\quad s<t\text{ in }[0,1].
\end{align*}
Recall that for any $\gamma\in AC^q([0,1],X)$, there exists a minimal a.e.~function $g\in L^q(0,1)$ satisfying the above, called {\it metric speed} denoted by $|\dot\gamma_t|$, which is defined as 
$|\dot\gamma_t|:=\lim_{h\downarrow0}{\sf d}(\gamma_{t+h}, \gamma_t)/h$ for $\gamma\in AC^q([0,1],X)$, 
$|\dot\gamma_t|:=+\infty$ otherwise. 
We define the kinetic energy functional $C([0,1],X)\ni\gamma\mapsto {\sf Ke}_t(\gamma):=\int_0^1|\dot{\gamma}_t|^q\d t$, if $\gamma\in AC^q([0,1],X)$, ${\sf Ke}_t(\gamma):=+\infty$ otherwise. 
For every $t\in[0,1]$, we define the {\it evaluation map} at time $t$ as follows: 
${\sf e}_t:C([0,1],X)\to X$, ${\sf e}_t(\gamma):=\gamma_t$ for $\gamma\in C([0,1]\to X)$. 
We easily see that ${\sf e}_t$ is a $1$-Lipschitz map.
\begin{defn}[\bf {\boldmath$q$}-test plan]\label{def:$q$-test}
{\rm Let $(X,{\sf d},\m)$ be a metric measure space and $q\in]1,+\infty[$.
A measure {\boldmath$\pi$}$\in\mathscr{P}(C([0,1],X))$ is said to be a $q$-test plan, provided 
\begin{enumerate}
\item[(i)]\label{item:qtest1} there exists $C>0$ so that $({\sf e}_t)_{\sharp}${\boldmath$\pi$}$\leq C\m$ for every $t\in[0,1]$;
\item[(ii)]\label{item:qtest2} we have $\int_{C([0,1],X)}{\sf Ke}_q(\gamma)${\boldmath$\pi$}$(\d\gamma)<\infty$. 
\end{enumerate} 
Moreover, we say that {\boldmath$\pi$} is an {\it $\infty$-test plan} if, instead of (ii), 
we have ${\rm Lip}(\gamma)\leq L$ for {\boldmath$\pi$}-a.e.~$\gamma$ (and thus for every $\gamma$ in the support of {\boldmath$\pi$} by the lower semi continuity  of the global Lipschitz constant with respect to uniform convergence).    
}
\end{defn}

\begin{defn}[{\bf Sobolev spaces {\boldmath$S^{\hspace{0.03cm}p}(X)$ and {\boldmath$W^{1,p}(X)$}}}]\label{def:Sobolev}
{\rm Let $(X,{\sf d},\m)$ be a metric measure space and $p\in]1,+\infty[$.
A Borel function $f$ belongs to $S^{\hspace{0.03cm}p}(X)$, 
provided there exists a $G\in L^p(X;\m)$,  
called {\it $p$-weak upper gradient} of $f$ so that 
\begin{align}
\int_{C([0,1],X)}|f(\gamma_1)-f(\gamma_0)|\text{\boldmath$\pi$}(\d\gamma)\leq \int_{C([0,1],X)}\int_0^1
G(\gamma_t)|\dot{\gamma}_t|\d t\text{\boldmath$\pi$}(\d\gamma),\quad \text{$\forall${\boldmath$\pi$} 
$q$-test plan}. \label{eq:SpSobolev}
\end{align} 
the set of $p$-weak upper gradient of a given Borel function $f$ is a closed convex subset of $L^p(X;\m)$. The minimal $p$-weak upper gradient, denoted by $|Df|_p$ is then the element of 
minimal $L^p$-norm in this class. Also, by making use of the lattice property of the set of $p$-weak 
upper gradient, such a minimality is also in the $\m$-a.e. sense (see \cite[Proposition~2.17 and Theorem~2.18]{Ch:metmeas}). 
It should be noted that $S^{\hspace{0.03cm}p}(X)\cap L^{\infty}(X;\m)$ is an algebra.
The Sobolev space,  
denoted by $W^{1,p}(X)$, is $L^p(X;\m)\cap S^{\,p}(X)$ as a set, equipped with the norm 
\begin{align*}
\|f\|_{W^{1,p}(X)}:=\left(\|f\|^p_{L^p(X;\m)}+\| |Df|_p\|^p_{L^p(X;\m)} \right)^{\frac{1}{p}} ,\quad f\in  W^{1,p}(X). 
\end{align*}
}
\end{defn}
\begin{remark}
{\rm The definition of $W^{1,p}(X)$ is based on the mass transport theory. 
There are various way to define $(1,p)$-Sobolev spaces by Cheeger, Shanmugalingam and etc. 
For general $p\in]1,+\infty[$, they are equivalent to each other (see 
\cite[Theorem~7.4]{AGS_Sobolev} for general $p\in]1,+\infty[$,  
\cite[Remark~2.27 and Theorem~2.28]{GPLecture} for the case $p=2$).   
}
\end{remark}
It is a standard fact that $(W^{1,p}(X),\|\cdot\|_{W^{1,p}(X)})$ is a Banach space in view of the weak lower semi continuity of weak upper gradient (see \cite[Theorem~2.7]{Ch:metmeas}). 
It should be noted that $W^{1,p}(X)\cap L^{\infty}(X;\m)$ is an algebra.  
It is in general false 
that $(W^{1,p}(X),\|\cdot\|_{W^{1,p}(X)})$ is reflexive and 
$p=2$ does not imply that $(W^{1,2}(X),\|\cdot\|_{W^{1,2}(X)})$ is a Hilbert space. 
When the latter situation occurs, we say that $(X,{\sf d},\m)$ is {\it infinitesimally Hilbertian} (see 
\cite{Gigli:OntheDifferentialStr}). Equivalently, we call $(X,{\sf d},\m)$  infinitesimally Hilbertian 
provided the following {\it parallelogram identity} holds:
\begin{align}
2|Df|_2^2+2|Dg|_2^2=|D(f+g)|_2^2+|D(f-g)|_2^2, \quad \m\text{-a.e.} \quad \forall f,g\in W^{1,2}(X).\label{eq:parallelogram}
\end{align}
Under the infinitesimally Hilbertian condition for $(X,{\sf d},\m)$, 
$(W^{1,p}(X),\|\cdot\|_{W^{1,p}(X)})$ 
becomes a uniformly convex Banach space for $p\in]1,+\infty[$, hence it is reflexive (see \cite[Proposition~4.4]{GigliNobili}). 
Denote by ${\rm Lip}(X)$ the family of Lipschitz functions on $(X,{\sf d})$ and let ${\rm lip}(f)(x):=\varlimsup_{y\to x}\frac{|f(y)-f(x)|}{{\sf d}(x,y)}$ if $x$ is an accumulation point. When $x$ is an isolated point, we set ${\rm lip}(f)(x):=0$. 
It is easy to see that $\{f\in {\rm Lip}(X)\mid {\rm lip}(f)\in L^p(X;\m)\}\subset S^{\,p}(X)$, 
consequently, $\{f\in {\rm Lip}(X)\cap L^p(X;\m)\mid  {\rm lip}(f)\in L^p(X;\m)\}\subset W^{1,p}(X)$, 
 because 
\begin{align*}
|f(\gamma_1)-f(\gamma_0)|\leq \int_0^1{\rm lip}(f)(\gamma_t)|\dot{\gamma}_t|\d t
\end{align*}
for any $\gamma\in C([0,1],X)$ (see \cite[(2.17)]{GPLecture}. In particular, ${\rm Lip}(X)_{bs}:=\{f\in {\rm Lip}(X)\mid {\rm supp}[f]\text{ is bounded }\}\subset W^{1,p}(X)$. 
 
Under \eqref{eq:parallelogram}, we can give a bilinear form $\langle D\cdot ,D\cdot\rangle:W^{1,2}(X)\times W^{1,2}(X)\to L^1(X;\m)$ which is defined by 
\begin{align*}
\langle Df,Dg\rangle:=\frac14|D(f+g)|_2^2-\frac14|D(f-g)|_2^2,\quad f,g\in W^{1,2}(X).
\end{align*}
Moreover, under the infinitesimally Hilbertian condition, 
the bilinear form $(\mathscr{E},D(\mathscr{E}))$ defined by 
\begin{align*}
D(\mathscr{E}):=W^{1,2}(X),\quad \mathscr{E}(f,g):=\frac12\int_X\langle Df,Dg\rangle\d \m
\end{align*}
is a strongly local Dirichlet form on $L^2(X;\m)$ admitting carr\'e du champ.

Denote by $(P_t)_{t\geq0}$ the semigroup on $L^2(X;\m)$ associated with $(\mathscr{E},D(\mathscr{E}))$. We call $(P_t)_{t\geq0}$ \emph{heat flow}.  Under \eqref{eq:parallelogram}, let $(\Delta, D(\Delta))$ be the 
$L^2$-generator associated with $(\mathscr{E},D(\mathscr{E}))$  defined by
\begin{align}
\left\{\begin{array}{rl} D(\Delta)&:=\{u\in D(\mathscr{E})\mid \text{there exists } w\in L^2(X;\m)\text{ such that }\\
&\hspace{3cm}\mathscr{E}(u,v)=-\int_Xwv\,\d\m\quad \text{ for any }\quad v\in D(\mathscr{E})\}, \\ \Delta u&:=w\quad\text{ for } w\in L^2(X;\m)\quad\text{specified as above,} \end{array}\right.\label{eq:generatorL2*}
\end{align}

\begin{defn}[{\bf {\boldmath${\sf RCD}(K,\infty)$}-spaces}]
{\rm A metric measure space $(X,{\sf d},\m)$ is said to be an {\it {\sf RCD}$(K,\infty)$-space} if 
it satisfies 
the following conditions: 
\begin{enumerate}
\item[\rm(1)]
$(X, {\sf d}, \m)$ is infinitesimally Hilbertian. 

\item[\rm(2)]
There exist $x_0 \in X$ and constants $c, C > 0$ such that 
$\m(B_r(x_0)) \le C \e^{c r^2}$. 

\item[\rm(3)]
If $f \in W^{1,2}(X)$ satisfies 
$| D f |_2 \le 1$ $\m$-a.e., then $f$ has a $1$-Lipschitz representative. 

\item[\rm(4)]
For any $f \in D ( \Delta )$ 
with $\Delta f \in W^{1,2}(X)$ 
and $g \in D ( \Delta ) \cap L^\infty (X; \m)$ 
with $g \ge 0$ and $\Delta g \in L^\infty (X; \m)$, 
\begin{align*}
\frac12 \int_X | D f |_2^2 \Delta g \, \d \m 
- \int_X \langle D f, D \Delta f \rangle g \, \d \m 
\ge 
K \int_X | D f |_2^2 g \, \d \m 
\end{align*}
\end{enumerate}
}
\end{defn}

Let ${\rm Test}(X)$ be the family of \emph{test functions} defined by 
 \begin{align*}
 {\rm Test}(X):=\{u\in D(\Delta)\cap L^{\infty}(X;\m)\mid \Gamma(u)^{\frac12}=|Du|_2\in L^{\infty}(X;\m)\text{ and }\Delta u\in D(\mathscr{E})\}. 
 \end{align*}
 It is shown in \cite[Lemma~3.2]{Sav14} that ${\rm Test}(X)$ forms an algebra under ${\sf RCD}(K,\infty)$-condition for $(X,{\sf d},\m)$.

\begin{thm}\label{thm:Coincidence}
Let $K\in\R$ and $p\in]1,+\infty[$. Suppose that $(X,{\sf d},\m)$ is an {\sf RCD}$(K,\infty)$-space. Then 
$(\mathscr{D}_{1,p},\|\cdot\|_{H^{1,p}})$ is closable on $L^p(X;\m)$ and 
its closure $(H^{1,p}(X),\|\cdot\|_{H^{1,p}})$ on $L^p(X;\m)$ coincides with $(W^{1,p}(X),\|\cdot\|_{W^{1,p}})$. 
\end{thm}
\begin{proof}[{\bf Proof}]
It is easy to see $(H^{1,2}(X),\|\cdot\|_{H^{1,2}})=(W^{1,2}(X),\|\cdot\|_{W^{1,2}})$. 
Suppose that {\sf RCD}$(K,\infty)$-condition holds for $(X,{\sf d},\m)$. Then, by \cite[Theorem~3.4]{GigliHan}, 
for any given $p,q\in ]1,+\infty[$, $u\in W^{1,p}(X)$ such that $u,|Du|_p\in L^q(X;\m)$ implies 
$u\in W^{1,q}(X)$ and $|Du|_p=|Du|_q$ $\m$-a.e. From this, we can conclude that 
$\mathscr{D}_{1,p}\subset W^{1,p}(X)$ with $\Gamma(u)^{\frac12}=|Du|_p$ $\m$-a.e., hence 
$\|u\|_{H^{1,p}}=\|u\|_{W^{1,p}}$ for any $u\in \mathscr{D}_{1,p}$. 
Since ($W^{1,p}(X),\|\cdot\|_{W^{1,p}})$ is a Banach space, the closability of $(\mathscr{D}_{1,p},\|\cdot\|_{H^{1,p}})$ on $L^p(X;\m)$ is clear.  
Thus, we have the inclusion 
$H^{1,p}(X)\subset W^{1,p}(X)$ with $\|u\|_{H^{1,p}}=\|u\|_{W^{1,p}}$ for $u\in H^{1,p}(X)$. 

 We claim that  ${\rm Test}(X)\cap W^{1,p}(X)$ is $W^{1,p}$-dense in $W^{1,p}(X)$. 
  For this, we show that 
 any $u\in W^{1,p}(X)$ can be $\|\cdot\|_{W^{1,p}}$-approximated by a sequence of $
 W^{1,p}(X)\cap L^{\infty}(X;\m)\cap L^2(X;\m)$. 
 By using the distance function ${\sf d}:X^2\to[0,+\infty[$, we can construct functions $\rho_n\in  {\rm Lip}(X)$ such that 
 $0\leq \rho_n\leq 1$ on $X$, $\rho_n=1$ on $B_n(o)$,  $\rho_n=0$ on $X\setminus B_{n+1}(o)$ and ${\rm lip}(\rho_n)=|D\rho_n|_p=|D\rho_n|_2\leq1$ $\m$-a.e.~on $X$. Then $u_n:=(-n)\lor (u\rho_n)\land n\in W^{1,p}(X)\cap L^{\infty}(X;\m)\cap L^2(X;\m)$ forms an $\|\cdot\|_{W^{1,p}}$-approximation to $u$. For any $u\in W^{1,p}(X)\cap L^{\infty}(X;\m)\cap L^2(X;\m)$, we see $P_tu\in {\rm Test}(X)\cap W^{1,p}(X)$ in view of 
 $|DP_tu|_p^p\leq e^{-pKt}P_t|Du|_p^p$ by \cite[Proposition~3.1]{GigliHan} and  
 $|DP_tu|_2^2\leq \frac{e^{K^-t}}{2t}\|u\|_{L^{\infty}(X;\m)}$ by \cite[Proposition~6.8]{ERST}. 
 From this, $\{P_tu\}$ is a $\|\cdot\|_{W^{1,p}}$-bounded sequence in ${\rm Test}(X)\cap W^{1,p}(X)$.  
 Thanks to the uniform convexity of $(W^{1,p}(X),\|\cdot\|_{W^{1,p}})$ and Banach-Saks type theorem by Kakutani~\cite{Kakutani}, any $u\in W^{1,p}(X)\cap L^{\infty}(X;\m)\cap L^2(X;\m)$ can be $\|\cdot\|_{W^{1,p}}$-approximated by a sequence in ${\rm Test}(X)\cap W^{1,p}(X)$. Thus the claim is confirmed. 
 
Applying \cite[Theorem~3.4]{GigliHan} again, we see the inclusion
 ${\rm Test}(X)\cap W^{1,p}(X)\subset \mathscr{D}_{1,p}\subset H^{1,p}(X)$. Therefore we obtain the 
 inclusion $W^{1,p}(X)\subset H^{1,p}(X)$ with $\|u\|_{H^{1,p}}=\|u\|_{W^{1,p}}$ for $u\in W^{1,p}(X)$.     
\end{proof}
}
\end{example}

\begin{example}[Configuration space over metric measure spaces]
{\rm Let $(M,g)$ be a complete smooth Riemannian manifold without boundary.   
The configuration space $\Upsilon$ over $M$ is the space of all locally finite point measures, that is, 
\begin{align*}
\Upsilon:=\{\gamma\in\mathcal{M}(M)\mid \gamma(K)\subset\N\cup\{0\}\quad \text{ for all compact sets}\quad K\subset M\}. 
\end{align*}
In the seminal paper Albeverio-Kondrachev-R\"ockner~\cite{AKR}, they identified a natural geometry on $\Upsilon$ by lifting the geometry of $M$ to $\Upsilon$. In particular, there exists a natural gradient $\nabla^{\Upsilon}$, divergence ${\rm div}^{\Upsilon}$ and 
Laplace operator $\Delta^{\Upsilon}$ on $\Upsilon$. It is shown in \cite{AKR} that the Poisson point measure $\pi$ on 
$\Upsilon$ is the unique (up to intensity) measure on $\Upsilon$ under which the gradient and divergence become 
dual operator in $L^2(\Upsilon;\pi)$. Hence, the Poisson measure $\pi$ is the natural volume measure on $\Upsilon$ 
and $\Upsilon$ can be seen as an infinite dimensional Riemannian manifold. The canonical Dirichlet form 
\begin{align*}
\mathscr{E}(F)=\int_{\Upsilon}|\nabla^{\Upsilon}F|_{\gamma}^2\pi(\d\gamma)
\end{align*}
 constructed in \cite[Theorem~6.1]{AKR} is quasi-regular and strongly local
 and it induces the heat semigroup $T_t^{\Upsilon}$ and a Brownian motion ${\bf X}^{\Upsilon}$ on $\Upsilon$ which can be identified with the independent infinite particle process. 
By Proposition~\ref{prop:closability}, one can construct the $(1,p)$-Sobolev space $(H^{1,p}(\Upsilon),\|\cdot\|_{H^{1,p}})$ for $p\in[2,+\infty[$ based on the quasi regular strongly local 
 Dirichlet form $(\mathscr{E}^{\Upsilon}, D(\mathscr{E}^{\Upsilon}))$ on $L^2(\Upsilon;\mu)$. 
 
 More generally, in Dello Schiavo-Suzuki~\cite{DelloSuzuki:ConfigurationI}, configuration space $\Upsilon$ over proper complete and separable metric space $(X,{\sf d})$ is considered. The configuration space $\Upsilon$ is endowed with the \emph{vague topology} $\tau_V$, induced by duality with continuous compactly supported functions on $X$, and with a reference Borel probability measure $\mu$ satisfying \cite[Assumption~2.17]{DelloSuzuki:ConfigurationI}, commonly understood as the law of a proper point process on $X$. In \cite{DelloSuzuki:ConfigurationI}, 
 they constructed the strongly local Dirichlet form $\mathscr{E}^{\Upsilon}$ defined to be the $L^2(\Upsilon;\mu)$-closure of a certain pre-Dirichlet form on a class of certain cylindrical functions and prove its quasi-regularity for a wide class of measures $\mu$ and base spaces (see \cite[Proposition~3.9 and Theorem~3.45]{DelloSuzuki:ConfigurationI}). 
 
 By Proposition~\ref{prop:closability}, one can construct the $(1,p)$-Sobolev space $(H^{1,p}(\Upsilon),\|\cdot\|_{H^{1,p}})$ for $p\in[2,+\infty[$ based on the quasi regular strongly local 
 Dirichlet form $(\mathscr{E}^{\Upsilon}, D(\mathscr{E}^{\Upsilon}))$ on $L^2(\Upsilon;\mu)$. 
 The contents of Theorem~\ref{thm:potential} are not known for the $(1,p)$-Sobolev spaces 
 based on $(\mathscr{E}^{\Upsilon}, D(\mathscr{E}^{\Upsilon}))$.
 
}
\end{example}

\bigskip
\noindent
{\bf Acknowledgment.} The author would like to thank Dr.~Ryosuke Shimizu for showing  
the references \cite{BV2005}, \cite{HinzRoecknerTepryaev} and his recent papers 
\cite{Shimizu,MurganShiumizu} during the preparation of this paper.  
He also thanks to Professor Lucian Beznea for showing his recent result \cite{Beznea} and its related literatures \cite{Beusekom,HoJacob2004,JacobSchlling2004}. Professor Michael Hinz kindly tells us his 
previous work \cite{HinzKochMeinert}.  

\bigskip

\noindent
{\bf Conflict of interest.} The authors have no conflicts of interest to declare that are relevant to the content of this article.

\bigskip
\noindent
{\bf Data Availability Statement.} Data sharing is not applicable to this article as no datasets were generated or analyzed during the current study.

%\emph{Acknowledgment.} The author would like to tell 
%their sincere gratitude to the anonymous referee. 
%His/Her comments help to improve the quality 
%of this paper very much.  

\providecommand{\bysame}{\leavevmode\hbox to3em{\hrulefill}\thinspace}
\providecommand{\MR}{\relax\ifhmode\unskip\space\fi MR }
\providecommand{\MRhref}[2]{%
  \href{http://www.ams.org/mathscinet-getitem?mr=#1}{#2}
}
\providecommand{\href}[2]{#2}

% \bibliographystyle{amsplain}
% \bibliography{refs}

\begin{thebibliography}{99}


\bibitem{AKR}
{S. Albeverio, Yu. G. Kondratiev and M. R\"ockner}, 
\emph{Analysis and geometry on configuration spaces}, J. Funct. Anal. {\bf 154} (1998), no. 2, 444--500.

%\bibitem{AM:AF}
%{S.~Albeverio and Z.-M. Ma},
%\emph{Additive functionals, nowhere Radon and Kato class smooth measures associated with Dirichlet forms},   
%Osaka J. Math. {\bf 29} (1992), no. 2, 247--265.
%
%\bibitem{AB}
%L.~Ambrosio and J.~Bertrand, 
%\emph{DC calculus},
%Math.~Z. {\bf 288} (2018), no.~3-4, 1037--1080.
%%Preprint. Available at \textsf{arXiv}: 1505.048 17. 

\bibitem{AES}
{L.~Ambrosio, M. Erbar and  G.~Savar\'e}, 
\emph{Optimal transport, Cheeger energies and contractivity of dynamic transport distances in extended spaces}, 
Nonlinear Analysis, {\bf 137} (2016), 77--134.

%\bibitem{AGMR}
%L.~Ambrosio, N.~Gigli, A.~Mondino and T.~Rajala,
%\emph{Riemannian Ricci curvature lower bounds in metric measure spaces with $\sigma$-finite measure}, 
%Trans.\ Amer.\ Math.\ Soc.\ {\bf 367} (2015), no.~7,  4661--4701.



%\bibitem{AGS_Calc}
%L.~Ambrosio, N.~Gigli and G.~Savar\'e,
%\emph{Calculus and heat flow in metric measure spaces and applications to spaces with Ricci bounds from below}, 
%Invent.\ Math.\ \textbf{195} (2014), no.~2, 289--391.

%\bibitem{AGS_Riem}
%\bysame,
%\emph{Metric measure spaces with Riemannian Ricci curvature bounded from 
%below}, 
%Duke Math.\ J.\ \textbf{163} (2014), no.~7, 1405--1490.

\bibitem{AGS_Sobolev}
{L.~Ambrosio, N.~Gigli and G.~Savar\'e},
\emph{Density of Lipschitz functions and equivalence 
of weak gradients in metric measure spaces}, 
Rev.\ Mat.\ Iberoam.\ {\bf 29} (2013), no.~3, 969--996.

%\bibitem{AGS_BakryEmery}
%\bysame,
%\emph{Bakry-\'Emery curvature-dimension condition and Riemannian Ricci curvature bounds}, Ann. Probab. {\bf 43} (2015), no.~1, 339--404. 

%\bibitem{AMS}
%{L.~Ambrosio, A.~Mondino, and G.~Savar\'e}, 
%\emph{Nonlinear diffusion equations and curvature conditions in metric measure spaces}, 
%%Preprint. To appear in 
%Memoirs Amer. Math. Soc. {\bf 262}, (2019), no. 1270.%, Available at \textsf{arXiv}: 1509.07273.
%
%\bibitem{ABD}
%{G. Antonelli, E. Bru\'e, and D. Semola}, 
%\emph{Volume bounds for the quantitative singular strata of non collapsed RCD metric measure spaces}, Anal. Geom. Metr. Spaces, {\bf 7} (2019), 158--178.

%\bibitem{Bakry1}
%{D.~Bakry}, 
%\emph{Etude des transformations de Riesz dans les vari\'et\'es riemaniennes \'a courbure 
%de Ricci minor\'ee}, S\'eminaire de Prob. XXI, Lecture Notes in Math. {\bf 1247}, Springer-Verr\lag, Berlin-Heidelberg-New York, (1987), 137--172.
%
%\bibitem{Bakry2}
%\bysame,
%\emph{On Sobolev and logarithmic Sobolev inequalities for Markov semi-groups}, 
%in \lq\lq New Trends in Stochastic Analysis\rq\rq\;  (K. D. Elworthy, S. Kusuoka and I. Shigekawa eds.), World Sci. Publishing, River Edge, NJ (1997), 43--75.
%
%\bibitem{BE1}
%{D.~Bakry and M.~\'Emery},
%        \emph{Diffusion hypercontractives}, in: S\'em. Prob. XIX, in: Lecture Notes in Math., vol. 1123, Springer-Verlag, Berlin/New York, 1985, pp.~177--206.

%\bibitem{BQ}
%{D.~Bakry and Z.-M.~Qian},
%\emph{Volume comparison theorems without Jacobi fields}, Current trends in potential theory, 115--122, Theta Ser. Adv. Math., {\bf 4}, Theta, Bucharest, 2005.
%Available at {\tt http://www.lsp.ups-tlse.fr/Bakry.}


%\bibitem{BGL_book}
%{D.~Bakry, I.~Gentil and M.~Ledoux},
%\emph{Analysis and geometry of Markov diffusion operators}, 
%Grundlehren der Mathematischen Wissenschaften {\bf 348}, Springer, Cham, 2014.

\bibitem{Beusekom}
{P. van Beusekom}, \emph{On nonlinear Dirichlet forms}, 
(PhD Thesis) 1994.

\bibitem{Beznea}
{C.~Beznea, L.~Beznea and M.~R\"ockner}, 
\emph{Nonlinear Dirichlet forms associated with quasi regular mappings}, preprint 2023, 
{\tt arXiv:2311.01585}  

\bibitem{BV2005}
{M.~Biroli and P.~G.~Vernole}, 
\emph{Strongly local nonlinear Dirichlet functionals and forms}, Adv. Math. Sci. Appl. {\bf 15} (2005), no. 2, 655--682.   

\bibitem{BjornBjorn}
{A.~Bj\"orn and J.~Bj\"orn},
\emph{Nonlinear potential theory on metric spaces}, EMS Tracts in Mathematics, {\bf 17} 2012.



\bibitem{BH}
{N.~Bouleau and F.~Hirsch}, 
           \emph{Dirichlet Forms and Analysis on Wiener Space},
           de Gruyter Studies in Mathematics, {\bf 14}. Walter de Gruyter \& Co., Berlin, 1991. 



%\bibitem{BorodinSalminen}
%{A.~N.~Borodin and P.~Salminen}, 
%\emph{Handbook of Brownian Motion-Facts and Formulae}, 
%Second edition. Probability and its Applications. Birkha\"user Verlag, Basel, 2002.
%

\bibitem{Braun:Tamed2021}
{M.~Braun}, 
\emph{Vector calculus for tamed Dirichlet spaces}, preprint 2022, to appear in Mem. Amer. Math. Soc. Available at {\tt  https://arxiv.org/pdf/2108.}{\tt 12374.pdf} 



\bibitem{CaoGuQiu}
{S.~Cao, Q.~Gu and H.~Qiu}, 
\emph{$p$-energies on p.c.f. self-similar sets}, Adv. Math. {\bf 405} (2022), Paper No. 108517, 58.

%\bibitem{Cav-Mil}
%{F.~Cavalletti and E.~Milman}, 
%\emph{The globalization theorem for the curvature dimension condition},
%Inventiones mathematicae {\bf 226} (2021), no.~1, 1--137.  

%Preprint. Available at \textsf{arXiv}: 1612.07623. 

%\bibitem{Cav-St}
%{F.~Cavalletti and K.-T. Sturm},
%\emph{Local curvature-dimension condition implies measure-contraction property}, 
%J.\ Funct.\ Anal.\ \textbf{262} (2012), no.~12,  5110--5127. 

%\bibitem{CFKZ:Pert} 
% {Z.-Q.~Chen, P.~J.~Fitzsimmons, K.~Kuwae and T.-S.~Zhang}, 
%\emph{Perturbation of symmetric Markov processes}, Probab. Theory Related Fields {\bf 140} 
%(2008), no.~1-2, 239--275.
%
%
%\bibitem{CFKZ:GenPert}   
%         \bysame, 
%\emph{On general perturbations of symmetric Markov processes}, J. Math. Pures et 
%Appliqu\'ees {\bf 92} (2009), no.~4, 363--374. 
%
%
% \bibitem{CFbook}
% {Z.-Q. Chen and M. Fukushima}, \emph{Symmetric Markov processes, time change, and
% boundary theory}, London Mathematical Society Monographs Series, {\bf 35}. Princeton University Press, Princeton, NJ, 2012.    
%
\bibitem{Ch:metmeas}
J.~Cheeger, \emph{Differentiability of {L}ipschitz functions on metric measure
  spaces}, Geom. Funct. Anal. \textbf{9} (1999), no.~3, 428--517.

%\bibitem{Choquet:Vol1}
%{G.~Choquet}, 
%\emph{Lectures on analysis. Vol. I: Integration and topological vector spaces}, Edited by J. Marsden, T. Lance and S. Gelbart W. A. Benjamin, Inc., New York-Amsterdam 1969 Vol. I

%\bibitem{Conway:FunctionalAnal}
%{J.~B.~Conway}, \emph{A Course in Functional Analysis}, Graduate Texts in Mathematics {\bf 96} 2nd ed.
%1985.  

%\bibitem{CoulhonDuong}
%{T.~Coulhon and X. T. Duong}, 
%\emph{Riesz transform and related inequalities on noncompact Riemannian manifolds}, 
%Comm. Pure Appl. Math. {\bf 56} (2003), 1728--1751.
%
%
%\bibitem{CKM}
%{M.~Cranston, W.~S.~Kendall, and P.~March},  
%\emph{The radial part of Brownian motion II. Its life and times on the cut locus}, 
%Probab. Th. Related Fields \textbf{96} (1993) no.~3, 353--368.

\bibitem{Day}
{M.~M.~Day}, 
\emph{Some more uniformly convex spaces}, 
Bull. Amer. Math. Soc. {\bf 47} (1941), 504--507.


% \bibitem{DMG:ComparisonCR}
%  {A.~Debiard, B.~Gaveau and E.~Mazet}, 
%  \emph{Th\'eor\`emes de comparaison en g\'eom\'etrie riemannienne} C. R. Acad. Sci. Paris S\'er. A-B {\bf 281} (1975), no. 12, Aii, A455--A458. 

% \bibitem{DMG:Comparison}
%  {A.~Debiard, B.~Gaveau and E.~Mazet}, 
%  \emph{Th\'eor\`emes de comparaison en g\'eom\'etrie riemannienne}, Publ. Res. Inst. Math. Sci. {\bf 12} (1976/77), no. 2, 391--425.


\bibitem{DelloSuzuki:LademacherSobolevToLip}
{L. Dello Schiavo and K. Suzuki}, 
\emph{Rademacher-type theorems and Sobolev-to-Lipschitz properties for strongly local Dirichlet spaces}, 
J. Funct. Anal. {\bf 281} (2021), no. 11, Paper No. 109234, 63 pp.

\bibitem{DelloSuzuki:ConfigurationI} 
\bysame, 
\emph{Configuration spaces over singular spaces — I. Dirichlet-Form and
Metric Measure Geometry}, arXiv:2109.03192, 2021. 
 
%\bibitem{DelloSuzuki:ConfigurationII} 
%\bysame, \emph{Configuration spaces over singular spaces — II. Curvature}, arXiv:2205.01379, 2021. 

%\bibitem{Dudley}
%{R.~M.~Dudley},
% \emph{Real Analysis and Probability}, Revised reprint of the 1989 original. Cambridge Studies in Advanced Mathematics, {\bf 74}. Cambridge University Press, Cambridge, 2002.
 
 
%\bibitem{EkelanfTemanm}
%{I.~Ekeland and R.~Temam}, 
%\emph{Convex analysis and variational problems}, Studies in Mathematics and its Applications, Vol. 1 (Second ed.). New York: North-Holland Publishing Co., Amsterdam-Oxford,American 1976. 
 
%\bibitem{EH}
%{M. Erbar and M. Huesmann}, 
%\emph{Curvature bounds for configuration spaces}, Calc. Var. Partial Differential Equations {\bf 54} (2015), no. 1, 397--430.
% 
%\bibitem{EKS}
%{M.~Erbar, K.~Kuwada and K.-T.~Sturm},
%\emph{On the equivalence of the entropic curvature-dimension condition and Bochner's inequality on metric measure spaces}, 
%Invent.\ Math.\ {\bf 201} (2015), no.~3, 993--1071.

\bibitem{ERST}
{M. Erbar, C. Rigoni, K.-T. Sturm and L. Tamanini},  
\emph{Tamed spaces — Dirichlet spaces with distribution-valued Ricci bounds}, 
 J. Math. Pures Appl. (9) {\bf 161} (2022), 1--69.
%preprint, 2020, 
%Avairable at {\tt arXiv:2009.03121, 2020.}


%\bibitem{EXK}
%{S.~Esaki, X.-Z. Jian and K.~Kuwae}, 
%\emph{The Littlewood-Paley-Stein inequality for tamed Dirichlet space by 
%distributional curvature lower bounds}, (2023) preprint.

\bibitem{EXKRiesz}
{S.~Esaki, X.-Z. Jian and K.~Kuwae}, 
\emph{Riesz transforms for Dirichlet spaces tamed by distributional curvature lower bounds}, (2023) preprint, {\tt  arXiv:2308.12728v1}

%\bibitem{Fitzsimmons}
%{P.~J.~Fitzsimmons}, 
%\emph{On the quasi-regularity of semi-Dirichlet forms},
%Potential Analysis {\bf 15} (2001), 151--185.

%\bibitem{Fuk:StrictDecomposition} 
%{M. Fukushima}, 
%\emph{On a strict decomposition of additive functionals for symmetric diffusion processes}, Proc. Japan Acad. {\bf 70} Ser. A (1994), no.~9, 277--281. 
% 
%\bibitem{Fuk:Semimartingale}
% {M.~Fukushima},
%\emph{On semi-martingale characterizations of functionals of symmetric Markov processes}, 
%Electron. J. Probab. {\bf 4} (1999), no.~18, 1--32.

\bibitem{FukKane}
{M.~Fukushima and H.~Kaneko}, 
\emph{On $(r,p)$-capacities for general Markovian semigroups}, In: Albeverio, S.
(Ed.), Infinite dimensional analysis and stochastic processes, Pitman, Research Notes Math. Boston (MA)
124 (1985), 41--47.  
  
\bibitem{FOT}
    {M.~Fukushima, Y.~Oshima and M.~Takeda},
   \emph{Dirichlet forms and symmetric Markov processes}, Second revised and extended edition. de Gruyter Studies in Mathematics, {\bf 19}. Walter de Gruyter \& Co., Berlin, 2011. 

% \bibitem{FST}
% {M.~Fukushima, K.~Sato and S.~Taniguchi}, 
% \emph{On the closable parts of pre-Dirichlet forms and the fine supports of underlying measures},
% Osaka J. Math. {\bf 28} (1991), no.~3, 517--535.

\bibitem{FukUe2004}
{M.~Fukuahima and T.~Uemura},
\emph{On spectral synthesis for contractive $p$-norms and Besov spaces}, Potential Anal. {\bf 20} (2004), 195--206. 

%\bibitem{G}
% {R.~K. Getoor},
%           \emph{Some remarks on measures associated with homogeneous
%           random measures}, Seminar on stochastic processes, 
%           1985, 94--107,
%           Birkh\"auser, Boston, 1986.

% \bibitem{Gigli:Splitting}
% {N.~Gigli}, 
% \emph{The splitting theorem in non-smooth context}, 
% Preprint. Available at \textsf{arXiv}:1302.5555v1.

% \bibitem{Gigli:AnOverview}
% {N.~Gigli}, 
% \emph{An overview of the proof of the splitting theorem in spaces with non-negative Ricci curvature}, 
% Anal. Geom. Metr. Spaces {\bf 2} (2014), 169--213. 

\bibitem{Gigli:OntheDifferentialStr}
{N.~Gigli}, 
\emph{
On the differential structure of metric measure spaces and applications}, 
Mem. Amer. Math. Soc. {\bf 236} (2015), no. 1113.

\bibitem{Gigli:NonSmoothDifferentialStr}
\bysame, 
\emph{Nonsmooth Differential Geometry--An Approach Tailored for Spaces with Ricci Curvature Bounded from Below}, 
Mem. Amer. Math. Soc. {\bf 251} (2018), no. 1196.

%\bibitem{Gigli:NonSmoothDifferentialStrKyoto}
%\bysame, 
%\emph{Lectures notes on differential calculus on RCD spaces}, preprint, 2017. 
%Available at {\tt https://cvgmt.sns.it/media/doc/paper/3373/RIMSnotes.pdf}
 


%\bibitem{Gigli:2013vi}
%{N.~Gigli, A.~Mondino, and T.~Rajala}, 
%\emph{{Euclidean spaces as weak tangents
%  of infinitesimally Hilbertian metric spaces with Ricci curvature bounded
%  below}}, J. Reine. Angew. Math. {\bf 705} (2015), 233--244. 
%
%
\bibitem{GigliHan}
{N.~Gigli and B.-X. Han}, 
\emph{Independence on $p$ of weak upper gradients on {\sf RCD} spaces}, 
J.\ Funct.\ Anal. {\bf 271} (2016), no.~1, 1--11.



\bibitem{GigliNobili}
{N.~Gigli and F. Nobili}, 
\emph{A first-order condition for the independence on $p$ of weak gradients}, 
J.\ Funct.\ Anal. {\bf 283} (2022), no.~11, 109686.

\bibitem{GPLecture}
{N.~Gigli and E.~Pasqualetto}, 
\emph{Lectures on Nonsmooth Differential Geometry},  Springer-Verlag, 2020. 

%\bibitem{GP}
%\bysame,
%%{N.~Gigli and E.~Pasqualetto}, 
%\emph{Behaviour of the reference measure on \textsf{RCD} spaces under charts}, 
%Communications in Analysis and Geometry {\bf 29} (2021), no.~6, 1391--1414.  
%%Preprint. Available at \textsf{arXiv}:1607.05188v2. 

%\bibitem{GigliTama}
%{N.~Gigli and L.~Tamanini}, 
%\emph{Second order differentiation formula on RCD${}^*(K,N)$ spaces}, Journal of the European Mathematical Society, {\bf 23} (2021), 1727--1795.


\bibitem{HermanPeironeStrichartz}
{P.~E.~Herman, R.~Peirone and R.~S.~Strichartz}, 
\emph{$p$-energy and $p$-harmonic functions on Sierpinski gasket type fractals}, 
Potential Anal. {\bf 20} (2004), no. 2, 125--148.

\bibitem{Hino:ASPM2004}
{M.~Hino}, \emph{Integral representation of linear functionals on vector lattices and its application to BV functions on Wiener space}, Adv. Stud. Pure Math., {\bf 41}
Mathematical Society of Japan, Tokyo, 2004, 121--140.

\bibitem{HinoFractal}
\bysame, \emph{Energy measures and indices of Dirichlet forms, with applications to derivatives on some fractals}, 
Proc. Lond. Math. Soc. (3) {\bf 100} (2010), no. 1, 269--302.




\bibitem{HinzRoecknerTepryaev}
{M.~Hinz, M.~R\"{o}ckner and A.~Teplyaev},
\emph{Vector analysis for Dirichlet forms and quasilinear PDE and SPDE on metric measure spaces},
Stochastic Process. Appl. {\bf 123} (2013), no. 12, 4373--4406.

\bibitem{HinzKochMeinert}
{M. Hinz, D. Koch and M. Meinert}, 
\emph{Sobolev spaces and calculus of variations on fractals}, 
Analysis, probability and mathematical physics on fractals, 419--450.
Fractals Dyn. Math. Sci. Arts Theory Appl., 5
World Scientific Publishing Co. Pte. Ltd., Hackensack, NJ, [2020].
%In: Analysis, Probability and Mathematical Physics on Fractals, pp. 419--450 (2020),

\bibitem{HoJacob2004}
{W.~Hoh and N.~Jacob}, 
\emph{Towards an $L^p$-potential theory for sub-Markovian semigroups:
variational inequalities and balayage theory}, J. Evol. Equ. {\bf 4} (2004), 297--312.

\bibitem{JacobSchlling2004}
{N.~Jacob and R.~Schilling},
\emph{Extended $L^p$ Dirichlet spaces}, International Mathematical Series  Volume 11, 
Around the research of Vladimir Maz'ya, Function Spaces
Edit. by Ari Laptev 221--238.
%
%
%\bibitem{Hsu:2001}
%{E.~P.~Hsu}, 
%\emph{Stochastic analysis on manifolds}, 
% Graduate Studies in Mathematics, {\bf 38}. American Mathematical Society, Providence, RI, 2002.  
%
%\bibitem{HuaKellXia}
%{B.~Hua, M.~Kell and C. Xia},
%\emph{Harmonic functions on metric measure spaces}, 
%Preprint. Available at \textsf{arXiv}:1308.3607v2.

% \bibitem{Jiang:Li-Yau}
% {R.~Jiang}, 
% \emph{The Li-Yau inequality and heat kernels on metric measure spaces},  J. Math. Pures Appl. (9) {\bf 104} (2015), no. 1, 29--57.

%\bibitem{Jiang:2016wo}
%{R.~Jiang, H.~Li and H.~Zhang}, 
%\emph{Heat kernel bounds on
%  metric measure spaces and some applications}, 
% Potential Anal. {\bf 44} (2016), no.~3, 601--627. 
  %Preprint. Available at
  %\textsf{arXiv}:1407.5289, 2014.

\bibitem{Kakutani}
{S.~Kakutani}, \emph{Weak convergence in uniformly convex spaces}, T\^ohoku Math. J. {\bf 45} (1938), 188--193.


%\bibitem{KawabiMiyokawa}
%{H.~Kawabi and T.~Miyokawa}, 
%\emph{The Littlewood-Paley-Stein inequality for diffusion processes on general metric spaces}, 
%J. Math. Sci. Univ. Tokyo {\bf 14} (2007), 1--30.


%\bibitem{Keith:measurable}
%{S.~Keith}, 
%Measurable differentiable structures and the Poincar\'e inequality, 
%\emph{Indiana Univ. Math. J}. {\bf 53} (2004), no. 4, 1127--1150. 
%
%\bibitem{KZ:Poincare}
%{S.~Keith and X. Zhong}, 
%The Poincar\'e inequality is an open ended condition,  
%Annals of Mathematics, {\bf 167} (2008), 575--599.
%
%\bibitem{KellMondino}
%{M.~Kell and A.~Mondino}, 
%\emph{On the volume measure of non-smooth spaces with Ricci curvature bounded below},  
%Annali della Scuola Normale Superiore di Pisa 
%{\bf 18} (2018), no.~2, 593--610.
%%Available at \textsf{arXiv}:1607.02036v1. 
%
%\bibitem{Kendall} 
%{W.~S.~Kendall}, 
%\emph{The radial part of Brownian motion on a manifold: a semimartingale property},
%\textit{Ann. Prob.} \textbf{15} (1987), no.~4, 1491--1500.
%

%\bibitem{KK:AnalChara}
% {D.~Kim and K.~Kuwae},  
% %\bysame,  
% \emph{Analytic characterizations of gaugeability for generalized Feynman-Kac functionals}, 
%Trans. Amer. Math. Soc. \textbf{369} (2017), no.~7, 4545--4596.

%\bibitem{Kitabeppu:2015wba}
%Y.~Kitabeppu and S.~Lakzian, \emph{{Characterization of low dimensional
%  $RC\!D^*(K,N)$ spaces}}, Anal. Geom. Metr. Spaces {\bf 4} (2016), 187--215. 
  %Preprint. Available at: \textsf{arXiv}:1505.00420v2.

%    \bibitem{Kw:reflect}
%        {K.~Kuwae}, 
%          \emph{Reflected Dirichlet forms and the uniqueness of Silverstein's extension}, 
%          Potential Anal. {\bf 16} (2002), no. 3, 221--247. 


% \bibitem{Kw:maximumprinciple}
%        {K.~Kuwae}, 
%         \emph{Maximum principles for subharmonic functions 
%         via local semi-Dirichlet forms},   
%        Canadian J. Math. {\bf 60} (2008), no. 4, 822-874. 


%\bibitem{KK}
%{K.~Kuwada and K.~Kuwae},
%\emph{Radial processes on $\mathsf{RCD}^* (K,N)$ spaces},
%Jour. Math. Pures et Appl. {\bf 126} (2019), no.~9, 72--108.
%

   \bibitem{Kw:func}
       {K.~Kuwae}, 
        \emph{Functional calculus for Dirichlet forms}, 
        \emph{Osaka J.~Math.~}{\bf 35} (1998), 683--715. 

%   \bibitem{Kw:Jensen}
%   \bysame, 
%   \emph{Jensen's inequality on convex spaces}, Calc. Var. Partial Differential Equations {\bf 49} (2014), no. 3-4, 1359--1378. 
%
%   \bibitem{Kw:ResolFlow}
%   \bysame, 
%   \emph{Resolvent flows for convex functionals and $p$-harmonic maps}, Anal. Geom. Metr. Spaces {\bf 3} (2015), no. 1, 46--72.
%      
%

%\bibitem{Kw:stochI}
%{K.~Kuwae}, 
% \emph{Stochastic calculus 
%over symmetric Markov processes without time reversal},  Ann. Probab. {\bf 38}  (2010),  
%no.~4, 1532--1569. 

%\bibitem{Kw:stochIErrata}
%\bysame, 
%\emph{Errata: Stochastic 
% calculus over symmetric Markov processes without time reversal},  Ann. Probab. {\bf 38}  (2010), no.~4,  1532--1569, Ann. Probab. {\bf 40}  (2012), 
%no.~6, 2705--2706. 

%  \bibitem{KMS:lap}
%      {K. Kuwae, Y. Machigashira and T. Shioya}, 
%         \emph{Sobolev spaces, Laplacian, and heat kernel 
%        on Alexandrov spaces}, {Math.~Z.} {\bf 238} 
%        (2001), no.~2, 269--316.
%
%\bibitem{KwSy:splitting}
%{K.~Kuwae and T.~Shioya}, 
%\emph{A topological splitting theorem for weighted
%  Alexandrov spaces}, 
%  Tohoku Math. J. {\bf 63} (2011), no~1, 59--76.

%\bibitem{LenglartLepinglePratelli}
%{E.~Lenglart, D.~L\'epingle and M. Pratelli}, 
%\emph{Pr\'esentation unifi\'ee de certaines in\'egalit\'es de th\'eorie des martingales}, 
%S\'eminaire de Prob. XIV, Lecture Notes in Math. {\bf 784}, Springer-Verlag, Berlin-Heidelberg-New York, (1980), 26--48.
%
% \bibitem{Xdli:RieszTrans}
% {X.-D.~Li}, 
% \emph{Riesz transforms for symmetric diffusion operators on complete Riemannian manifolds}, 
% Rev. Mat. Iberoam. {\bf 22} (2006), no. 2, 591--648.

% \bibitem{Xdli:Liouville}
% \bysame, 
% \emph{Liouville theorems for symmetric diffusion operators on complete Riemannian manifolds},  J. Math. Pures Appl. (9) {\bf 84} (2005), no. 10, 1295--1361.

% \bibitem{LiX:PerelmanEntropyFor}
% \bysame, 
% \emph{Perelman's entropy formula for the Witten Laplacian on Riemannian manifolds via Bakry-Emery Ricci curvature}, 
% Math. Ann. {\bf 353} (2012), no. 2, 403--437.

%\bibitem{LV2}
%{J.~Lott and C.~Villani},
%\emph{Ricci curvature for metric-measure spaces via optimal transport}, 
%Ann.\ of Math.\ {\bf 169} (2009), no.~3, 903--991.

%\bibitem{Loomis:harmonicAnal}
%{L.~H.~Loomis}, 
%\emph{An introduction to abstract harmonic analysis}, D. Van Nostrand Company, Inc., Toronto-New York-London, 1953. 

\bibitem{MR}
{Z.-M.~Ma and M.~R\"ockner}, 
\emph{Introduction to the Theory of (Non-Symmetric) Dirichlet Forms}, 
Springer-Verlag, Berlin-Heidelberg-New York, 1992.

%\bibitem{Meyer}
%{P.~A.~Meyer}, 
%\emph{D\'emonstration probabiliste de certaines in\'egalit\'es de Littlewood-Paley}, S\'eminaire de Prob. X, Lecture Notes in Math. {\bf 511}, Springer-
%Verlag, Berlin-Heidelberg-New York, (1976), 125--183.

%\bibitem{MondinoNaberI}
%{A. Mondino and A. Naber},
%\emph{Structure theory of metric-measure spaces with lower Ricci curvature bounds}, {\bf 21} (2019), no.~6, 1809--1854.

%Preprint. To appear in Jour. European Math Soc. Available at \textsf{arXiv}:1405.2222.
   
\bibitem{MurganShiumizu}
{M.~Murgan and R.~Shimizu}, 
\emph{First-order Sobolev spaces, self-similar energies and energy measures on the Sierpi\'nski carpet}, 
preprint, 2023, {\tt  arXiv:2308.06232v1}    
   
\bibitem{Na} 
{S.~Nakao},
\emph{Stochastic calculus for continuous additive functionals of zero energy},
Z.~Wahrsch.~verw. Gebiete {\bf 68} (1985) 557--578.   
   
% \bibitem{Rajala:Poincare}
% {T.~Rajala}, 
% \emph{Local Poincar\'e inequalities from stable curvature conditions on metric spaces},  Calc. Var. {\bf 44} (2012), no.~3, 477--494. 

%\bibitem{Okura}
%{H.~$\widehat{\rm O}$kura},
%\emph{A new approach to the skew product of symmetric Markov processes}, Mem. Fac. Eng. and Design Kyoto Inst. Tech. {\bf 46} (1998), 1--12.

%\bibitem{Ohta:Convex}
%{S. Ohta},
%\emph{Convexities of metric spaces}, Geom. Dedicata {\bf 125} (2007), 225--250. 

%\bibitem{Ohta}
%{S. Ohta},
%\emph{On the measure contraction property of metric measure spaces},  
%Comment. Math. Helv. \textbf{82} (2007) no.~4, 805--828. 

%\bibitem{Petrunin:LottVillani}
%{A. Petrunin}, 
%\emph{Alexandrov meets Lott--Villani--Sturm}, M\"unster J. of Math. {\bf 4} 
%(2011), 53--64.
%
%\bibitem{Rajala:Poincare2}
%{T.~Rajala}, 
%\emph{Interpolated measures with bounded density in metric spaces satisfying the curvature-dimension conditions of Sturm}, J. Funct. Anal. {\bf 263} (2012) no.~4, 896--924. 
   
%\bibitem{Max:Comparison}
%{M.~K. von Renesse},
%\emph{Heat kernel comparison on Alexandrov spaces with curvature bounded below}, Potential Anal. {\bf 21} (2004), no.~2, 151--176.

\bibitem{RoecknerSchmuland}
{M.~R\"ockner and B.~Schmuland},
\emph{Tightness of general  $C^{1,p}$ capacities on Banach space}, 
J. Funct. Anal. {\bf 108} (1992), no. 1, 1--12.

%
%\bibitem{Royden:RealAnal}
%{H.~L.~Royden}, 
%\emph{Real analysis}, Third edition. Macmillan Publishing Company, New York, 1988. 

 \bibitem{Sav14}
 {G.~Savar\'e}, 
 \emph{Self-improvement of the Bakry-\'Emery condition and
 Wasserstein contraction of the heat flow in $RC\!D (K, \infty)$ metric measure spaces}, 
 Discrete and Contin.\ Dyn.\ Syst.\ \textbf{34} (2014), 1641--1661.

\bibitem{SavExtendedMetric}
\bysame, 
\emph{Sobolev spaces in extended metric-measure spaces}, in New Trends on Analysis and Geometry in Metric Spaces, Lecture Notes in Math. 2296. 117--276, 2017.   
 
 
%\bibitem{Schwarz}
%{G.~Schwarz}, 
%\emph{Hodge decomposition, --- a method for solving boundary value problems}, 
%Lecture Notes in Mathematics, 1607. Springer Verlag, Berlin, 1995. viii+155 pp.  

%\bibitem{Shanmugalingam:Newton}
%      {N.~Shanmugalingam}, 
%     \emph{Newtonian spaces: an extension of Sobolev spaces to metric measure spaces}, 
%     Rev. Mat. Iberoamericana {\bf 16} (2000), no. 2, 243--279. 
%\bibitem{Shanmugalingam:Harmonic}
%      {N.~Shanmugalingam}, 
%\emph{Harmonic functions on metric spaces}, Illinois J. Math. {\bf 45} (2001), no. 3, 1021--1050. 

\bibitem{RneSchilling}
{R.~Schilling}, 
\emph{Measuers, integrals and martingales}, Cambridge University Press, 2005. 

%\bibitem{Shigekawa1}
%{I.~Shigekawa}, 
%\emph{Littlewood-Paley inequality for a diffusion satisfying the logarithmic Sobolev inequality and for the Brownian motion on a Riemannian
%manifold with boundary}, Osaka J. Math. {\bf 39} (2002), 897--930.
%
\bibitem{ShigekawaText}
{I.~Shigekawa}, 
\emph{Stochastic Analysis}, Translations of Mathematical Monographs {\bf 224}. Iwanami Series in Modern Mathematics. American Mathematical Society, Providence, RI, 2004.

%\bibitem{ShigekawaYoshida}
%{I.~Shigekawa and N.~Yoshida}, 
%\emph{Littlewood-Paley-Stein inequality for a symmetric diffusion}, 
%J. Math. Soc. Japan {\bf 44} (1992), 251--280.


% \bibitem{Schu:Fatou}
% {B.~Schmuland}, 
% \emph{Positivity preserving forms have the Fatou property}, 
% Potential Anal. {\bf 10} (1999), no.~4, 373--388.

%\bibitem{SchoenYau:LectDiffGeo}
% {R.~Schoen and S.~T.~Yau}, 
% \emph{Lectures on differential geometry}, 
% Conference Proceedings and Lecture Notes in Geometry and Topology, I. International Press, Cambridge, MA, 1994.

\bibitem{Shimizu}
{R.~Shimizu}, 
\emph{Construction of $p$-energy and associated energy measures on Sierpi\'nski carpets}, preprint, 2023, to appear in Transaction of AMS, {\tt  arXiv:2110.13902v3} 

% \bibitem{Staff:Liouville}
% {S.~Stafford},
% \emph{A probabilistic proof of S.-Y. Cheng's Liouville theorem},  
% Ann. Probab. \textbf{18} (1990), no. 4, 1816--1822. 

%\bibitem{Stein}
%{E.~M.~Stein}, 
%\emph{Topics in Harmonic Analysis Related to Littlewood-Paley Theory}, Annals of Math. Studies {\bf 63}, Princeton Univ. Press,1970.

%  \bibitem{St:DirI}
%         K.-Th.~Sturm:
%       \emph{Analysis on local Dirichlet spaces I,-Recurrence, conservativeness
%         and $L^p$-Liouville properties-}
%        J.~Reine.~Angew.~Math. {\bf 456} (1994), 173--196.
        


%   \bibitem{St:DirII}
%         {K.-Th.~Sturm}, 
%        \emph{Analysis on local Dirichlet spaces. II. 
%       Upper Gaussian estimates for the fundamental 
%       solutions of parabolic equations},
%       Osaka J.~Math. {\bf 32} (1995), 275--312.

%\bibitem{Sturm:1995gu}
%K.-T. Sturm, \emph{{Sharp estimates for capacities and applications to
%  symmetric diffusions}}, Probab. Theory Related Fields \textbf{103} (1995),
%  no.~1, 73--89.

%\bibitem{StI} 
%{K.-T.~Sturm},
%\emph{On the geometry of metric measure spaces.~I},
%Acta Math.\ {\bf 196} (2006), 65--131.
% 
%\bibitem{StII} 
%\bysame,
%\emph{On the geometry of metric measure spaces.~II},
%Acta Math.\ {\bf 196} (2006), 133--177.
%
%\bibitem{Sturm_coupling} 
%\bysame, \emph{Metric measure spaces with variable Ricci bounds and couplings of Brownian motions}. Festschrift Masatoshi Fukushima, 553--575, 
%Interdiscip. Math. Sci., 17, World Sci. Publ., Hackensack, NJ, 2015. 
%
%\bibitem{Sznitzman}
%{A.~S.~Sznitman}, 
%\emph{Brownian motion and obstacles and random media}, Springer-Verlag Berlin Heidelberg 1998. 
%
% \bibitem{book_Vil1}
% C.~Villani, \emph{Topics in optimal transportations}, Graduate studies in 
% mathematics, 58, American mathematical society, Providence, RI, 2003.

%\bibitem{Vi2} 
%{C.~Villani}, \emph{Optimal transport, old and new},  Springer-Verlag, Berlin, 2009.

%\bibitem{Yamaguchi:Luminy}
%{T. Yamaguchi}, \emph{A convergence theorem 
%in the geometry of Alexandrov spaces}, Actes
%de la Table Ronde de G\'eom\'etrie Diff\'erentielle (Luminy, 1992), S\'emin. Congr., vol. 1, Soc. Math. France, Paris, 1996, pp. 601--642.

%\bibitem{YoshidaNobuo}
%{N.~Yoshida}, 
%\emph{Sobolev spaces on a Riemannian manifold and their equivalence}, 
%J. Math. Kyoto Univ. {\bf 32} (1992), 621--654.

%\bibitem{Yosida}
%{K.~Yosida}, 
%\emph{Functional Analysis}, Sixth Edition. Springer-Verlag. Berlin Heidelberg New York 1980.
%


\end{thebibliography}
\end{document}